\documentclass[3p,12pt]{elsarticle}
\usepackage{psfig}
\usepackage{latexsym,amssymb,amsmath,amsthm,amsfonts}

\newtheorem{thm}{Theorem}
\newtheorem{lem}{Lemma}
\theoremstyle{definition}
\newtheorem{remk}{Remark}
\newtheorem{defn}{Definition}
\newtheorem{expl}{Example}

\journal{Electronic Journal of Probability}

\begin{document}

\begin{frontmatter}

\title{\textbf{On convergence of general wavelet decompositions\\ of nonstationary stochastic processes}\\[5mm]
\small \textit{Short title: Uniform convergence of wavelet decompositions}}

\author[ku]{Yuriy Kozachenko}
\ead{ykoz@ukr.net}

\author[la]{Andriy Olenko\corref{cor1}}
\ead{a.olenko@latrobe.edu.au}

\author[ku]{Olga Polosmak}
\ead{DidenkoOlga@yandex.ru}

\cortext[cor1]{Corresponding author. Phone: +61-3-9479-2609 \quad  Fax:  +61-3-9479-2466}
\address[ku]{Department of Probability Theory, Statistics and Actuarial Mathematics,\\ Kyiv University, Kyiv, Ukraine}
\address[la]{Department of Mathematics and Statistics, La Trobe University,\\ Victoria 3086, Australia}

\begin{abstract}
The paper investigates uniform convergence of wavelet expansions of Gaussian random processes. The convergence is obtained under simple general conditions on processes and wavelets which can be easily verified. Applications of the developed technique are shown for several classes of stochastic processes. In particular, the main theorem is adjusted to the fractional Brownian motion case. New results on the rate of convergence of the wavelet expansions in the space $C([0,T])$ are also presented.
\end{abstract}

\begin{keyword}
Convergence in probability \sep Uniform convergence \sep Convergence rate \sep Gaussian process  \sep Fractional Brownian motion   \sep Wavelets


\MSC 60G10 \sep 60G15 \sep 42C40
\end{keyword}

\end{frontmatter}

\section{Introduction}

In the book \cite{har} wavelet expansions of non-random functions bounded on $\mathbb R$ were studied in different
spaces. However, developed deterministic methods may not be appropriate to investigate wavelet expansions of stochastic processes. For example, in the majority of cases, which are interesting from theoretical and practical application points of view, stochastic processes have almost surely unbounded sample paths on $\mathbb R.$ It indicates the necessity of elaborating special stochastic techniques.

Recently, a considerable attention was given to the properties of the wavelet orthonormal series representation of
random processes. More information on convergence of wavelet expansions of random processes in various spaces, references and numerous applications can be found in \cite{Cambanis, didier, ist, kozol1, kozol2, kozol3, kozpol, Meyer, kur, zha}. Most known stochastic results concern the mean-square  or almost sure convergence, but for various practical applications one needs to require uniform convergence. To
give an adequate description of the performance of wavelet approximations in both cases, for points where the processes are relatively smooth and points where spikes occur, we can use the uniform distance instead of global integral $L^p$ metrics.  A more in depth discussion, further references and various applications in econometrics, simulations of stochastic processes and functional data analysis can be found in \cite{del,egg,fan,kur,phi}. In his 2010 Szekeres Medal inauguration speech, an eminent leader in the field, Prof. P. Hall stated the development of uniform stochastic approximation methods as one of frontiers in modern functional data analysis.

Figures~1 and 2 illustrate some features of wavelet expansions of stochastic processes. Figure~1 presents a simulated realization of the Wiener process  and its wavelet reconstructions by two sums with different numbers of terms. The figure has been generated by the R package \textsc{wmtsa} \cite{wmtsa}. Besides providing a realization of the Wiener process and its wavelet reconstructions, we also plot corresponding reconstruction errors. Figure~2 shows maximum absolute reconstruction errors for 100 simulated realizations. To reconstruct each realization of the Wiener process two approximation sums (as in Figure~1) were used.  We clearly see that empirical probabilities of obtaining large reconstruction errors become smaller if the number of terms in the wavelet expansions increases. Although this effect is expected, it has to be established theoretically in a stringent way for different classes of stochastic processes and wavelet bases. It is also important to obtain theoretical estimations of the rate of convergence for such stochastic wavelet expansions.

\begin{center}
\begin{minipage}{4.8cm}
 {\psfig{figure=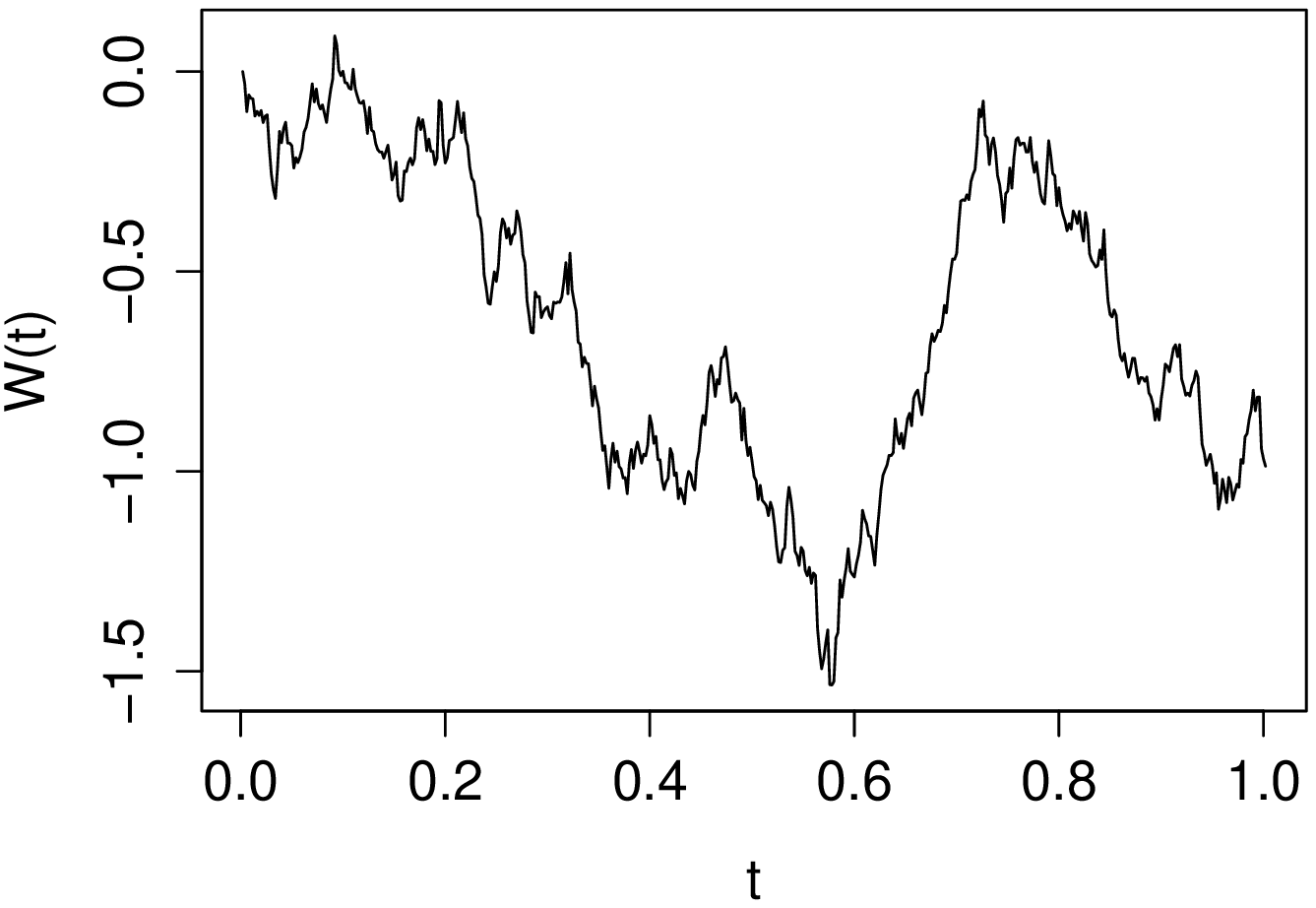 ,width=4.8cm}}  \end{minipage}
\begin{minipage}{4.8cm}
{\psfig{figure=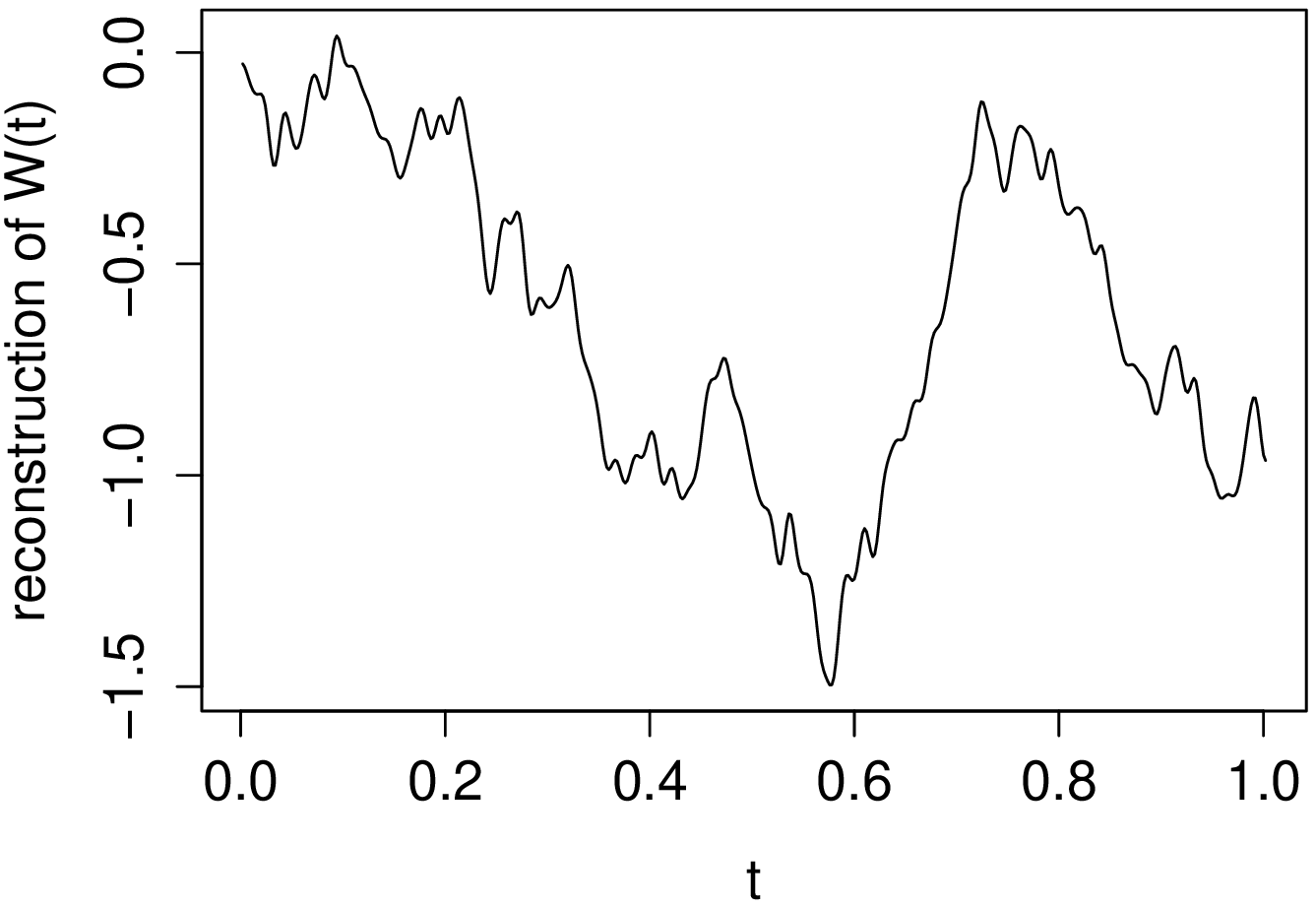,width=4.8cm}} \end{minipage}
\begin{minipage}{4.8cm}
{\psfig{figure=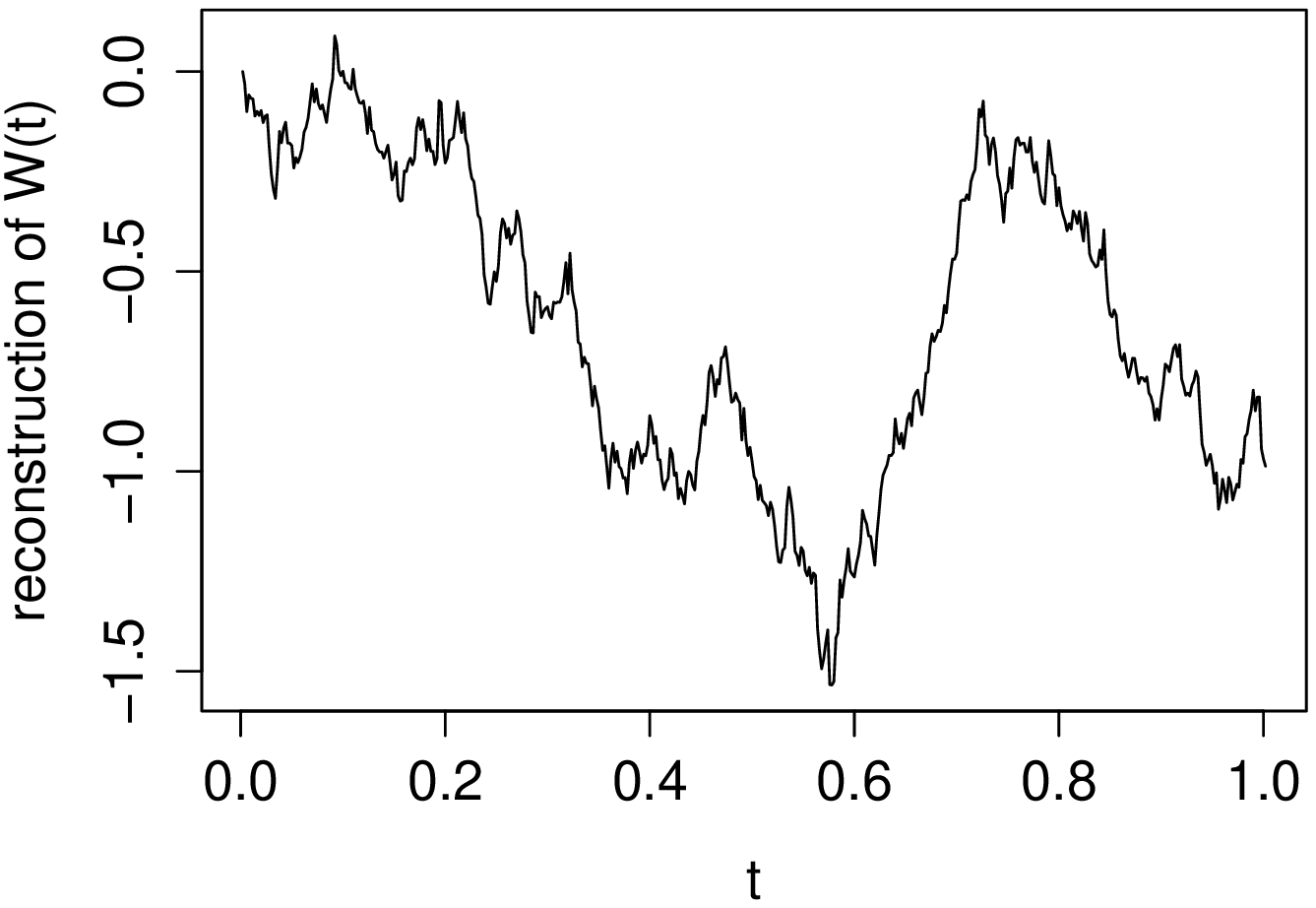,width=4.8cm}} \end{minipage}\\
{\textbf{Fig.\,1}\ \ \small Plots of  the Wiener process and its wavelet reconstructions}
\end{center}
\begin{center}
\begin{minipage}{5.5cm}
 {\psfig{figure=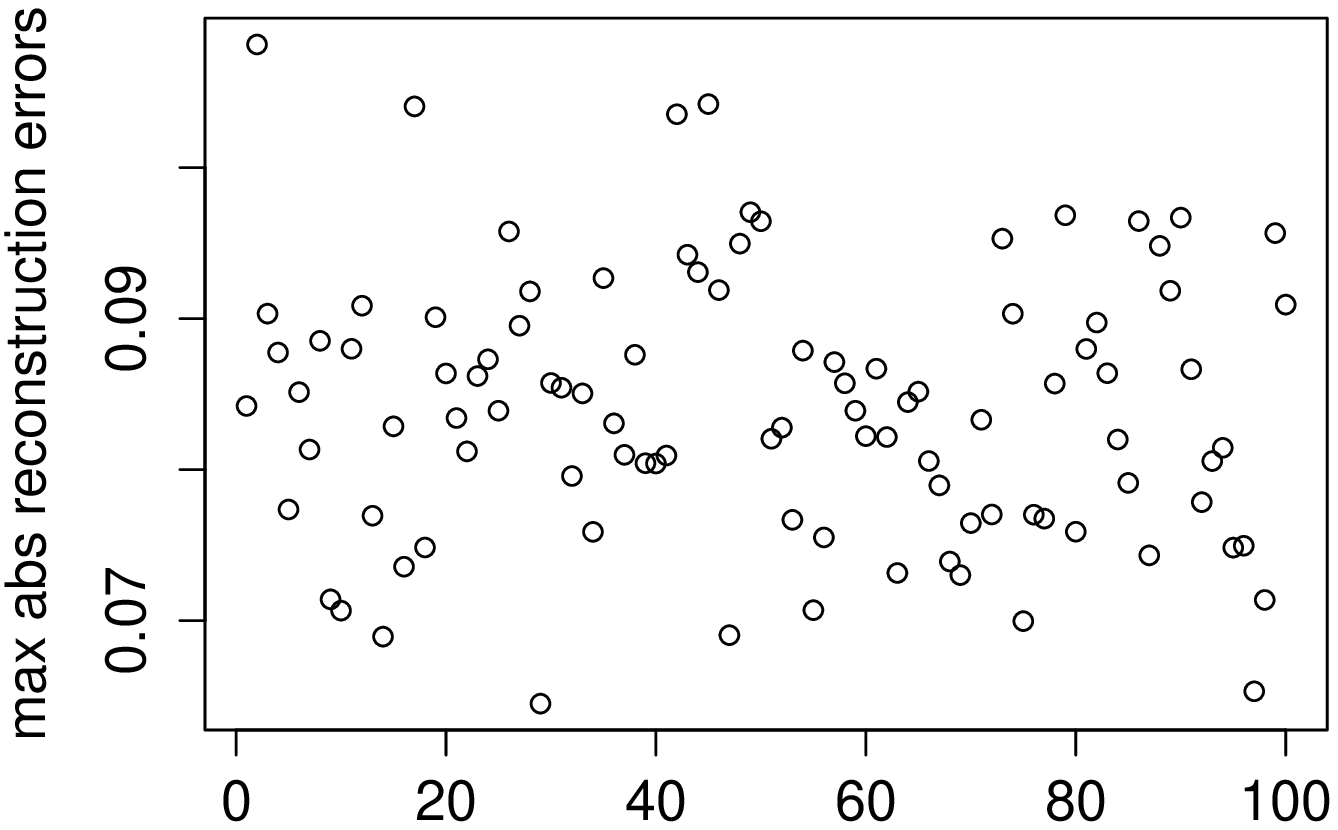 ,width=5.5cm}}  \end{minipage}\quad\quad
\begin{minipage}{5.5cm}
{\psfig{figure=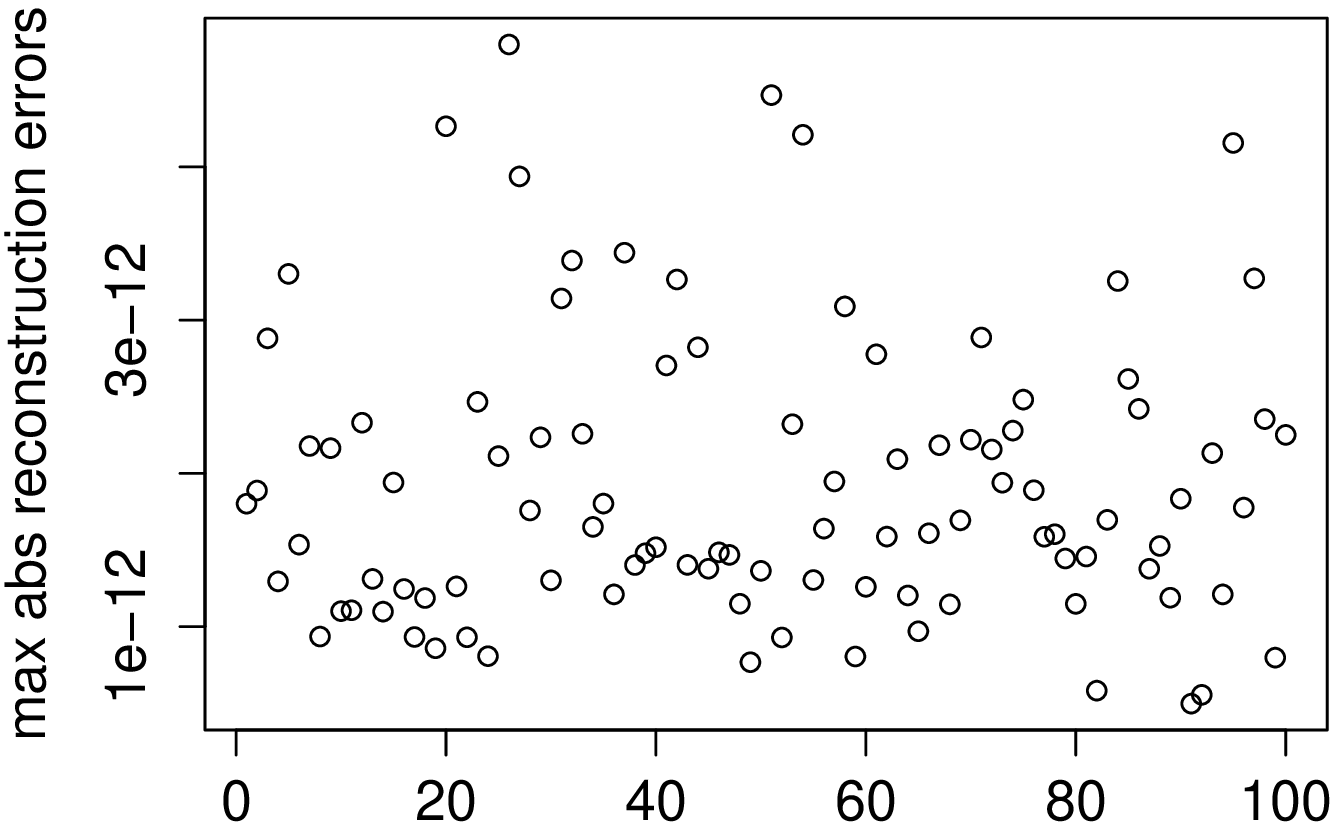,width=5.5cm}} \end{minipage}\\
{\textbf{Fig.\,2}\ \ \small Plots of reconstruction errors for 100 simulated realizations}
\end{center}

In this paper we make an attempt to derive general results on stochastic uniform convergence which are valid for wavelet expansions of wide classes of stochastic processes. The paper deals with the most general class of such wavelet expansions in comparison with particular cases considered by different authors, see, for example, \cite{aya,Cambanis,dzha,igl,kur}. Applications of the main theorem to special cases of practical importance (stationary processes, fractional Brownian motion, etc.) are demonstrated. We also prove the exponential rate of convergence of the wavelet expansions.

Throughout the paper, we impose minimal assumptions on the wavelet bases. The results are obtained under simple conditions which can be easily verified. The conditions  are weaker than those in the former literature. 

These are novel results on stochastic uniform convergence of general finite wavelet expansions of nonstationary random processes. The specifications of established results are also new (for example, for the case of stationary stochastic processes, compare  \cite{kozol1, kozol2}).

Finally, it should be mentioned that the analysis of the rate of convergence gives a constructive algorithm for determining the number of terms in the wavelet expansions to ensure the uniform approximation of stochastic processes with given accuracy. It provides a practical way to obtain explicit bounds on the sharpness of finite wavelet series approximations.

The organization of the article is the following. In the second section we introduce the necessary background from wavelet
theory and certain sufficient conditions for mean-square convergence of wavelet expansions in the space $L_2(\Omega).$
In \S 3 we formulate and discuss the main theorem on uniform convergence in probability of the wavelet expansions of Gaussian random processes. The next section contains the proof of the main theorem. Two applications of the developed technique are shown in section 4.  In \S 5 the main theorem is adjusted to the fractional Brownian motion case. Lastly, we  obtain the rate of convergence of the wavelet expansions in the space $C([0,T]).$

In what follows we use the symbol $C$ to denote constants which are not important for our discussion. Moreover, the same
symbol $C$ may be used for different constants appearing in the same proof.

\section{Wavelet representation of random processes}

Let $\phi(x),$ $x\in\mathbb R,$ be a function from the space
$L_2(\mathbb R)$ such that $\widehat{\phi}(0)\ne 0$ and
$\widehat{\phi}(y)$ is continuous at $0,$ where
$$\widehat{\phi}(y)=\int_{\mathbb
R}e^{-iyx}{\phi(x)}\,dx$$ is the Fourier transform of $\phi.$

 Suppose that the following assumption
holds true:
$$\sum_{k\in\mathbb Z} |\widehat{\phi}(y+2{\pi}k)|^2=1\  {\rm (a.e.)}
$$

There exists a function $m_0(x)\in L_2([0,2\pi])$, such that
$m_0(x)$ has the period $2\pi$ and
$$\widehat{\phi}(y)=m_0\left(y/2\right)\widehat{\phi}\left(y/2\right)\ {\rm (a.e.)}
$$
 In this case the  function $\phi(x)$ is called
the $f$-wavelet.

Let $\psi(x)$ be the inverse Fourier transform of the function
$$\widehat{\psi}(y)=\overline{m_0\left(\frac
y2+\pi\right)}\cdot\exp\left\{-i\frac
y2\right\}\cdot\widehat{\phi}\left(\frac y2\right).$$ Then the
function
$$\psi(x)=\frac{1}{2\pi}\int_{\mathbb
R}e^{iyx}{\widehat{\psi}(y)}\,dy$$ is called the $m$-wavelet.

Let
\begin{equation}\label{psi_jk}
\phi_{jk}(x)=2^{j/2}\phi(2^jx-k),\quad
\psi_{jk}(x)=2^{j/2}\psi(2^jx-k),\quad j,k \in\mathbb Z\,,
\end{equation}
where $\phi(x)$ and $\psi(x)$ are defined as above.

The conditions $\widehat{\phi}(0)\ne 0$ and
$\widehat{\phi}(y)$ is continuous at $0$ guarantee the completeness of  system (\ref{psi_jk}), see \cite{her}.
Then the family of functions $\{\phi_{0k};
\psi_{jk},\,j\in \mathbb N_0,\,k\in\mathbb Z\},$ $\mathbb N_0:=\{0,1,2,...\},$ is an orthonormal basis in
$L_2(\mathbb R)$ (see, for example, \cite{har}. A necessary and sufficient condition on $m_0$ is given in \cite{chu,dau}).

\begin{defn} Let $\delta(x),x\in\mathbb R,$ be a function from the space $L_2(\mathbb R).$
If the system of functions $\left\{\delta(x-k),\,k\in \mathbb Z \right\}$ is an
orthonormal system, then the function $\delta(x)$ is called a
scaling function.
\end{defn}

\begin{remk} $f$-wavelets and $m$-wavelets are scaling functions.
\end{remk}

\begin{defn}  Let $\delta(x),x\in\mathbb R,$ be a scaling function.
If there exists a bounded function $\Phi(x), \ x\ge0,$
such that $\Phi(x)$ is a decreasing function, $|\delta(x)|\le\Phi(|x|)$ a.e. on $\mathbb R,$
and $\int_{\mathbb R}
\left(\Phi(|x|)\right)^{\gamma}\,dx<\infty$ for some $\gamma>0,$  then $\delta(x)$ satisfies
assumption $S'(\gamma).$ 
\end{defn}

\begin{remk} If assumption $S'(\gamma)$ is satisfied for some
$\gamma>0$ then the assumption is also true for all
$\gamma_1>\gamma.$
\end{remk}

\begin{lem}\label{lem0}
If for some function $\Phi_\phi(\cdot)$ the $f$-wavelet $\phi$ satisfies assumption $S'(\gamma),$  then there exist a function $\Phi_\psi(\cdot)$ such that for the corresponding $m$-wavelet $\psi$ assumption $S'(\gamma)$ holds true and $\Phi_\psi(x)\ge \Phi_\phi(x),$ for all $x\in \mathbb R.$

\end{lem}
\begin{proof}
 The wavelets $\phi$ and $\psi$ admit the following representations, see \cite{har},
\[
	\phi (x)=\sqrt{2} \sum \limits _{k\in \mathbb{Z}}h_{k} \phi \left(2x-k\right),
\]
\begin{equation}\label{eq_04}
	\psi(x)=\sqrt{2} \sum \limits _{k\in \mathbb{Z}}\lambda_{k} \phi \left(2x-k\right),
\end{equation}
where $h_{k}=\sqrt{2} \int_{R}\phi \left(u\right) \overline{\phi \left(2u-k\right)}du,$  $\sum \limits _{k\in \mathbb{Z}}\left|h_{k} \right|^{2} <\infty,$ and $\lambda _{k} =(-1)^{1-k} \overline{h_{1-k}}$.

If $k\geq 0,$ then
\[|h_{k}| \leq 	\sqrt{2} \int_{-\infty }^{\frac{k}{3}}\Phi_\phi \left(\left|u\right|\right)\Phi_\phi \left(\left|2u-k\right|\right)du +\sqrt{2} \int_{\frac{k}{3}}^{\infty }\Phi_\phi \left(u\right)\Phi_\phi \left(\left|2u-k\right|\right)du 
\]
\[\leq \sqrt{2}\Phi_\phi \left({k}/{3} \right) \int_{-\infty }^{\frac{k}{3} }\Phi_\phi \left(\left|u\right|\right)du +\sqrt{2}\Phi_\phi \left({k}/{3} \right) \int_{\frac{k}{3} }^{\infty }\Phi_\phi \left(\left|2u-k\right|\right)du 
\]
\[\leq \sqrt{2} \Phi \left({k}/{3} \right)\left(\int_{\mathbb{R}}\Phi_\phi \left(\left|u\right|\right)du +\int_{\mathbb{R}}\Phi_\phi \left(\left|2u-k\right|\right)du \right)=C\Phi_\phi \left({k}/{3} \right),
\]
where
\[C: =\frac{3}{\sqrt{2}} \int_{\mathbb{R}}\Phi_\phi \left(\left|u\right|\right)du.
\]

Similarly, for $k\leq 0$ we get $\left|h_{k} \right|\leq C\Phi_\phi \left({\left|k\right|}/{3} \right).$

Thus, for all $k\in \mathbb{Z},$
\begin{equation}\label{eq_05}
	\left|h_{k} \right|\le C \, \Phi_\phi \left({\left|k\right|}/{3} \right).
\end{equation}

Note that the series in the right-hand side of (\ref{eq_04}) converges in the $L_2(\mathbb R)$-norm. Therefore, there exists a subsequence of partial sums which converges to $\psi \left(x\right)$ a.e. on $\mathbb R.$ Thus, by (\ref{eq_04}) and (\ref{eq_05}) we obtain 
\[
	\left|\psi \left(x\right)\right|\leq \sqrt{2} C \sum \limits _{k\in \mathbb{Z}}\Phi_\phi \left(\left|2x-k\right|\right) \cdot \Phi_\phi \left(\frac{\left|1-k\right|}{3} \right)=
\]
\begin{equation}\label{eq_06}
	=\sqrt{2} C \sum \limits _{k\in \mathbb{Z}}\Phi_\phi \left(\left|2x-1-k\right|\right) \cdot \Phi_\phi \left({\left|k\right|}/{3} \right)=:I\left(2x-1\right)
\end{equation}
a.e.  on $\mathbb R.$

If $u:=2x-1,$ then for $u>0$
\[
	I(u)\leq \sqrt{2} C
	\left(\sum \limits_{\left|k\right|\leq \frac{3}{4} u}\Phi_\phi \left(\left|u-k\right|\right) \Phi_\phi \left({\left|k\right|}/{3} \right)\right.
\]
\[
	+\left.\sum \limits _{\left|k\right|\geq \frac{3}{4} u}\Phi_\phi \left(\left|u-k\right|\right) \cdot \Phi_\phi \left({\left|k\right|}/{3} \right)\right)
	=:\sqrt{2}C  \left(A_{1} \left(u\right)+A_{2} \left(u\right)\right).
\]

Notice also that
\[
		A_{1} (u)\leq \Phi_\phi \left({u}/{4} \right)\cdot \sum \limits _{\left|k\right|\leq \frac{3}{4} u}\Phi_\phi \left({\left|k\right|}/{3} \right) \leq \Phi_\phi \left({u}/{4} \right)\cdot \sum \limits _{k\in \mathbb{Z}}\Phi_\phi \left({\left|k\right|}/{3} \right),\]
\[
		A_{2} (u)\leq \Phi_\phi \left({u}/{4} \right)\cdot \sum \limits _{\left|k\right|\geq \frac{3}{4} u}\Phi_\phi \left(\left|u-k\right|\right) \leq \Phi \left({u}/{4} \right)\cdot \sum \limits _{k\in \mathbb{Z}}\Phi_\phi \left(\left|u-k\right|\right).
\]

Therefore, for $u>0$
\begin{equation}\label{eq_09}
	I(u)\leq \sqrt{2} C \Phi_\phi \left({u}/{4} \right) \left(\sum \limits _{k\in \mathbb{Z}}\Phi_\phi \left(\left|u-k\right|\right) +\sum \limits _{k\in\mathbb{Z}}\Phi_\phi \left({\left|k\right|}/{3} \right)\right).
\end{equation}

We are to prove that, $\sum \limits _{k\in \mathbb{Z}}\Phi_\phi \left(\left|u-k\right|\right) +\sum \limits _{k\in \mathbb{Z}}\Phi_\phi \left({\left|k\right|}/{3} \right)$ is bounded.

Note that
\[
	\sum \limits _{k\geq u+1}\Phi_\phi \left(k-u\right)
	\leq \sum \limits _{k\geq u+1}\int_{k-1}^{k}\Phi_\phi \left(v-u\right) dv\leq  \int_{0}^{\infty }\Phi_\phi \left(v\right)dv <\infty,
\]
and, similarly,
\[
	\sum \limits _{k\leq u-1}\Phi_\phi \left(u-k\right) \leq \int_{0}^{\infty }\Phi_\phi \left(v\right)dv.
\]
Thus,
\[
	\sum \limits _{k\in \mathbb{Z}}\Phi_\phi \left(\left|u-k\right|\right) \leq 2\left(\Phi_\phi \left(0\right)+\int_{0}^{\infty }\Phi_\phi \left(v\right)dv \right)=:D_{1}.
\]
Since
\[
	\sum \limits _{k=1}^{\infty }\Phi_\phi \left({k}/{3} \right) \leq \sum \limits _{k=1}^{\infty }\int_{k-1}^{k}\Phi_\phi \left({v}/{3} \right)dv =
3\int _{0}^{\infty }\Phi_\phi (u)du,
\]
it follows that
\[
	\sum \limits _{k=1}^{\infty }\Phi_\phi \left({\left|k\right|}/{3} \right)\leq \Phi_\phi \left(0\right)+6\int_{0}^{\infty }\Phi_\phi \left(u\right)du =:D_{2}\,.
\]

Thus, we conclude that for $u>0$
\begin{equation}\label{eq_12}
	I(u)\leq \sqrt{2}C \left(D_{1} +D_{2} \right) \Phi_\phi \left({u}/{4} \right).
\end{equation}

Since for $u<0$
$$\Phi_\phi \left(\left|u+k\right|\right)=\Phi_\phi \left(\left|-u-k\right|\right)=\Phi_\phi \left(\left|\left|u\right|-k\right|\right),$$
it follows that
\begin{equation}\label{eq_13}
	I(u)\leq  \sqrt{2} C \left(D_{1} +D_{2}\right)\Phi_\phi \left({|u|}/{4} \right)
\end{equation}
for every $u\in \mathbb R.$

By (\ref{eq_06}), (\ref{eq_09}),  (\ref{eq_12}),  and (\ref{eq_13}), 
\begin{equation}\label{eqF}
	\left|\psi (x)\right|\leq \tilde{C}\, \Phi_\phi \left(\left|\frac{2x-1}{4} \right|\right) \ \mbox{a.e.  on}\ \mathbb R.
\end{equation}

The desired result  follows from (\ref{eqF}) if we chose
$$\Phi_\psi(x):=	\begin{cases}\max(1,\tilde{C}), & x \in (0,1/2),\\
\max(1,\tilde{C})\cdot\max\left(\Phi_\phi(x),\Phi_\phi \left(\left|\frac{2x-1}{4} \right|\right)\right), & x \not\in (0,1/2).
\end{cases}$$
\end{proof}

Motivated by Lemma~\ref{lem0}, we will use the following assumption instead of two separate assumptions $S'(\gamma)$ for the $f$-wavelet
$\phi$ and the $m$-wavelet $\psi.$\vspace{2mm}

{\bf Assumption $S(\gamma).$} {\it For some function $\Phi(\cdot)$ and $\gamma>0$ both the father and
mother wavelets satisfy assumption $S'(\gamma).$}\vspace{2mm}

Let $\{\Omega, \cal{B}, P\}$ be a standard probability space. Let
$\mathbf{X}(t),$ $t\in\mathbb R,$ be a  random process such that
$\mathbf E\mathbf{X}(t)=0, \ \mathbf E|\mathbf{X}(t)|^{2}<\infty.$

If sample trajectories of this process are  in the space
$L_2(\mathbb R)$ with probability one, then it is possible to obtain the
representation (wavelet representation)
\begin{equation}\label{X(t)}\mathbf{X}(t)=\sum_{k\in \mathbb Z}\xi_{0k}\phi_{0k}(t)+\sum_{j=0}^{\infty}\sum_{k\in \mathbb Z}
\eta_{jk}\psi_{jk}(t)\,, \end{equation}
where
\begin{equation}\label{intksieta}\xi_{0k}=\int_{\mathbb R}\mathbf{X}(t)\overline{\phi_{0k}(t)}\,dt,\quad
\eta_{jk}=\int_{\mathbb
R}\mathbf{X}(t)\overline{\psi_{jk}(t)}\,dt\,.\end{equation}

The majority of random processes does not possess the required property. For
example, sample paths of stationary processes are not in the space
$L_2(\mathbb R)$ (a.s.). However, in many cases it is possible to construct a representation of type~(\ref{X(t)}) for $\mathbf{X}(t).$

Consider the approximants of $\mathbf{X}(t)$ defined by
\begin{equation}\label{x_nk}
\mathbf{X}_{n,\mathbf{k}_n}(t)=\sum_{|k|\le
k_0'}\xi_{0k}\phi_{0k}(t)+\sum_{j=0}^{n-1}\sum_{|k|\le
k_j}\eta_{jk}\psi_{jk}(t)\,,
\end{equation}
where $\mathbf{k}_n:=(k_0',k_0,...,k_{n-1}).$

Theorem~\ref{8103} below guarantees the mean-square convergence of
$\mathbf{X}_{n,\mathbf{k}_n}(t)$  to $\mathbf{X}(t)$ if
$k_0'\to\infty,$ $k_j\to\infty,$ $j\in \mathbb N_0,$ and $n\to\infty.$
The latter means that we increase the number $n$ of multiresolution analysis subspaces which are used to approximate  $\mathbf{X}(t).$ For each multiresolution analysis subspace  $j=0',0,1,2...$ the number $k_j$ of its basis vectors,  which are used in the approximation,   increases too, as $n$ tends to infinity. Thus, for each fixed $k$ and $j$ there is $n_0\in \mathbb N_0$ that the terms $\xi_{0k}\phi_{0k}(t)$ and $\eta_{jk}\psi_{jk}(t)$ are included in all $\mathbf{X}_{n,\mathbf{k}_n}(t)$ for $n\ge n_0$ (i.e., each $\xi_{0k}\phi_{0k}(t)$ and $\eta_{jk}\psi_{jk}(t)$ can be absent only in the finite number of $\mathbf{X}_{n,\mathbf{k}_n}(t)$).

\begin{thm}\label{8103} {\rm\cite{kozpol}} Let $\mathbf{X}(t),$ $t\in\mathbb R,$ be a random process
such that $\mathbf E\mathbf{X}(t)=0,$ $\mathbf
E|\mathbf{X}(t)|^2<\infty$ for all $t\in\mathbb R,$ and its
covariance function  $R(t,s):=\mathbf E\mathbf{X}(t)\overline{\mathbf{X}(s)}$  is continuous. Let the $f$-wavelet
$\phi$ and the $m$-wavelet $\psi$ be continuous functions which satisfy
assumption $S(1).$ Let $c(x),$ $x\in\mathbb R,$ denote a non decreasing on $[0,\infty)$
even function with $c(0)>0.$
Suppose that there exists a function $A:(0,\infty)\to (0,\infty),$ 
and $a_0,\,x_0\in (0,\infty)$ such that  $ c(ax)\le c(x)\cdot A(a),$ for all $a\ge a_0,$
$x\ge x_0.$\vspace{1mm}

If
$$\int_{\mathbb R}c(x)\Phi(|x|)\,dx<\infty\quad {\rm and} \quad
|R(t,t)|^{1/2}\le c(t),$$ then

\begin{enumerate}
\item[\rm 1.]  $\mathbf{X}_{n,\mathbf{k}_n}(t)\in L_2(\Omega)\,;$
\item[\rm 2.]  $\mathbf{X}_{n,\mathbf{k}_n}(t)\to \mathbf{X}(t)$ in mean square when  $n\to\infty,$ $k_0'\to\infty,$ and $k_j\to\infty$ for all
$j\in \mathbb N_0\,.$
\end{enumerate}
\end{thm}

\section{Uniform convergence  of wavelet expansions
for Gaussian random processes}
In this section we show that, under suitable conditions, the sequence $\mathbf{X}_{n,\mathbf{k}_n}(t)$ converges in probability in Banach space $C([0,T])$, $T>0,$ i.e.
$$P\left\{\sup_{0\le t\le T} |\mathbf{X}(t)-\mathbf{X}_{n,\mathbf{k}_n}(t)|>\varepsilon \right\}\to 0,
$$
when  $n\to\infty,$ $k_0'\to\infty,$ and $k_j\to\infty$ for all $j\in \mathbb{N}_0:=\{0,1,...\}\,.$ More details on the general theory of random processes in the space $C([0,T])$ can be found in \cite{bulkoz}.

\begin{thm}\label{213}{\rm \cite{kozsli}}  Let $\mathbf{X}_n(t),$ $t\in[0,T],$
 be a sequence of Gaussian stochastic processes such that
all  $\mathbf{X}_n(t)$ are separable in $([0,T], \rho),$ $\rho(t,s)=|t-s|,$ and
$$\sup_{n\ge 1}\sup_{|t-s|\le h}\left( \mathbf E|\mathbf{X}_n(t)-\mathbf{X}_n(s)|^2\right)^{1/2}\le \sigma(h)\,,$$
where $\sigma(h)$ is a function, which is monotone increasing in a neighborhood of the origin and
$\sigma(h)\to 0$ when $h\to 0\,.$

Assume that for some $\varepsilon>0$
\begin{equation}\label{142}\int_0^\varepsilon
\sqrt{-\ln\left(\sigma^{(-1)}(u)\right)}\,du<\infty\,,\end{equation}
where $\sigma^{(-1)}(u)$ is the inverse function of $\sigma(u)$.
 If the random variables $\mathbf{X}_n(t)$ converge in probability to the random variable
 $\mathbf{X}(t)$ for all $t\in [0,T]$, then $\mathbf{X}_n(t)$ converges  to  $\mathbf{X}(t)$ in the space
$C([0,T]).$
\end{thm}
\begin{remk}\label{772}For example, it is easy to check that assumption~(\ref{142}) holds true for
$$\sigma(h)=\frac C{\ln^{\beta}\left(e^{\alpha}+
\frac1{h}\right)}\quad \mbox {and} \quad \sigma(h)=C
h^{\kappa},$$  when $C>0,$
$\beta>1/2,$ $\alpha>0,$ and $\kappa>0.$
\end{remk}

The following theorem is the main result of the paper.
\begin{thm}\label{mainuniform} Let  a Gaussian process $\mathbf{X}(t),$ $t\in\mathbb R,$  its covariance function, the $f$-wavelet $\phi,$
and the corresponding $m$-wavelet $\psi$  satisfy the assumptions
of Theorem~{\rm\ref{8103}}.

Suppose that
\begin{enumerate}
\item[{\rm(i)}] assumption $S(\gamma), \gamma\in(0,1),$  holds true for $\phi$ and
$\psi;$
\item[{\rm(ii)}] the  integrals
$\int_{\mathbb
R}\ln^{\alpha}(1+|u|)|\widehat{\psi}(u)|\,du$ and
$\int_{\mathbb
R}\ln^{\alpha}(1+|u|)|\widehat{\phi}(u)|\,du$ converge for some
$\alpha>1/{2(1-\gamma)};$
  \item[{\rm(iii)}] there exist
constants $b_0$ and $c_j,$ $j\in \mathbb N_0,$ such that for all integer $k$  $\mathbf E|\xi_{0k}|^2\le b_0,$ $\mathbf
E|\eta_{jk}|^2\le c_j,$   and
\begin{equation}\label{sum_cj}\sum\limits_{j=0}^{\infty}\sqrt{c_j}2^{\frac
j2}j^{\alpha(1-\gamma)}<\infty.
\end{equation}
\end{enumerate}
Then $\mathbf{X}_{n,\mathbf{k}_n}(t)\to \mathbf{X}(t)$
uniformly in probability on each interval $[0,T]$ when  $n\to\infty,$ $k_0'\to \infty,$
and $k_j\to\infty$ for all $j\in \mathbb N_0\,.$
\end{thm}

\begin{remk}
If both wavelets $\phi$ and $\psi$  have compact supports, then some assumptions of
Theorem~\ref{mainuniform} are superfluous. In the following
theorem we give an example by considering approximants of the form
$$\mathbf{X}_{n}(t):=\sum_{k\in\mathbb Z}\xi_{0k}\phi_{0k}(t)+\sum_{j=0}^{n-1}\sum_{k\in\mathbb Z}
\eta_{jk}\psi_{jk}(t)\,.
$$

\end{remk}

\begin{thm}\label{th4} Let $\mathbf{X}(t),$ $t\in\mathbb R,$ be
a separable centered Gaussian random process such that its covariance function  $R(t,s)$ is continuous. Let the $f$-wavelet $\phi$
and the corresponding $m$-wavelet $\psi$ be continuous functions
with compact supports and the integrals  $\int_{\mathbb
R}\ln^{\alpha}(1+|u|)|\widehat{\psi}(u)|\,du$ and
$\int_{\mathbb
R}\ln^{\alpha}(1+|u|)|\widehat{\phi}(u)|\,du$ converge for some $\alpha>1/{2(1-\gamma)},$ $\gamma\in(0,1).$  If there exist
constants $c_j,$ $j\in \mathbb N_0,$ such that $\mathbf
E|\eta_{jk}|^2\le c_j$ for all $k\in \mathbb{Z},$ and
assumption {\rm(\ref{sum_cj})} is satisfied, then
$\mathbf{X}_{n}(t)\to \mathbf{X}(t)$ uniformly in
probability on each interval $[0,T]$ when  $n\to\infty.$
\end{thm}

\begin{proof} The assumptions of Theorem~\ref{8103} and $S(\gamma), 0<\gamma<1,$  are satisfied because   $\phi$ and $\psi$ have compact supports. Therefore, the desired result follows from  Theorem~\ref{mainuniform}.
\end{proof}

\begin{remk}
For example, Daubechies wavelets satisfy the assumptions of Theorem~\ref{th4}.
\end{remk}

\section{Proof of the main theorem}
To prove Theorem~\ref{mainuniform} we need some auxiliary results.

\begin{lem}\label{lem1} If $\delta(x)$ is a scaling function satisfying assumption $S'(\gamma),$ then
\begin{equation}\label{supSgamma}\sup\limits_{x\in\mathbb
R}S_{\gamma}(x)\le 3\Phi^\gamma(0)+4\int_{1/2}^{\infty}\Phi^{\gamma}(t)\,dt,
\end{equation}
where
$$S_{\gamma}(x)=\sum\limits_{k\in\mathbb Z}|\delta(x-k)|^{\gamma}.$$
\end{lem}
\begin{proof}
The lemma is a simple generalization of a result from \cite{har}.

Since $S_{\gamma}(x)$ is a periodic function with period
1,  it is sufficient to prove~(\ref{supSgamma}) for
$x\in[0,1].$

Notice, that for $x\in[0,1]$ and integer $|k|\ge 2$ the inequality $|x-k|\ge |k|/2$ holds true.  Hence, $\Phi(|x-k|)\le \Phi\left(|k|/2\right)$ and
$$S_{\gamma}(x)\le \Phi^{\gamma}(|x|)+\Phi^{\gamma}(|x+1|)
+\Phi^{\gamma}(|x-1|)+\sum\limits_{|k|\ge2}\Phi^{\gamma}\left({|k|}/2\right)$$
$$\le
3\Phi^{\gamma}(0)+2\sum\limits_{k=2}^{\infty}\,\int_{k-1}^{k}\Phi^{\gamma}\left(
t/2\right)\,dt=3\Phi^{\gamma}(0)+4\int_{\frac
12}^{\infty}\Phi^{\gamma}(t)\,dt.$$
\end{proof}

\begin{lem}\label{lem2} Let $\hat\delta(x)$ denote the Fourier transform of the scaling function\ $\delta(x),$ $\delta_{jk}(x):=2^{\frac
j2}\delta(2^jx-k),$ $j\in \mathbb N_0,$ $k\in\mathbb Z.$ If for  some
$\alpha>0$
\[\int_{\mathbb
R}\ln^{\alpha}(1+|u|)|\widehat{\delta}(u)|\,du<\infty,
\]
then for all $x,y\in \mathbb R$ and $k\in\mathbb Z$
\[|\delta_{0k}(x)-\delta_{0k}(y)|\le
R_{0\alpha}\cdot\ln^{-\alpha}\left(e^{\alpha}+\frac{1}{|x-y|}\right),
\]
\[|\delta_{jk}(x)-\delta_{jk}(y)|\le 2^{\frac
j2}j^{\alpha}R_{1\alpha}\cdot
\ln^{-\alpha}\left(e^{\alpha}+\frac{1}{|x-y|}\right),\ \  j\in \mathbb N_0,
\]
where $R_{0\alpha}:=\frac{1}{\pi}\int_{\mathbb
R}\ln^{\alpha}\left(e^{\alpha}+\frac{|u|}2\right)|\widehat{\delta}(u)|\,du$
and $R_{1\alpha}:=\frac{1}{\pi}\int_{\mathbb
R}\ln^{\alpha}\left(e^{\alpha}+|u|+1\right)|\widehat{\delta}(u)|\,du.$
\end{lem}
\begin{proof}
Since $\delta(x)=\frac{1}{2\pi}\int_{\mathbb
R}e^{ixu}\hat\delta(u)\,du,$ we have
\begin{eqnarray}\label{intsin}|\delta_{jk}(x)-\delta_{jk}(y)|
\le\frac{2^{\frac j2}}{2\pi}\int_{\mathbb R}
\left|e^{ixu2^j}-e^{iyu2^j}\right||\hat\delta\left(
u\right)|\,du=\frac{2^{\frac j2}}{\pi}\int_{\mathbb R}
\left|\sin{\frac{(x-y)u2^j}2}\right||\hat\delta\left( u\right)|\,du.
\end{eqnarray}
Note that for $v\ne0$ the following inequality holds:
\begin{eqnarray}\label{sinln}\left|\sin\left(\frac{s}v\right)\right|\le
\frac{\left|\ln\left(e^{\alpha}+|s|\right)\right|^{\alpha}}
{\left|\ln\left(e^{\alpha}+|v|\right)\right|^{\alpha}}\,.
\end{eqnarray}
By~(\ref{intsin}) and~(\ref{sinln}) we obtain
\begin{eqnarray*}|\delta_{jk}(x)-\delta_{jk}(y)|\le
\frac{2^{\frac
j2}}{\pi\ln^{\alpha}\left(e^{\alpha}+\frac1{|x-y|}\right)}\int_{\mathbb
R} {\ln^{\alpha}\left(e^{\alpha}+\frac
{|u|2^j}2\right)}|\hat\delta\left( u\right)|\,du.
\end{eqnarray*}
The assertion of the lemma follows from this inequality.
\end{proof}

\begin{lem}\label{lemsumdelta} If a scaling function $\delta(x)$ satisfies the assumptions of
Lemmata~{\rm\ref{lem1}} and {\rm\ref{lem2}}, then for
$\gamma\in(0,1)$ and  $\alpha>0:$
\begin{eqnarray}\label{sumdeltajk-}\sum\limits_{k\in\mathbb
Z}|\delta_{jk}(x)-\delta_{jk}(y)|\le 2^{\frac{j}{2}+1}j^{\alpha(1-\gamma)}R_1^{1-\gamma}(\alpha)
\left(3\Phi^{\gamma}(0)+4\int_{1/2}^{\infty}\Phi^{\gamma}(t)\,dt\right)\nonumber\\
\times\left(\ln\left(e^{\alpha}+\frac{1}{|x-y|}\right)\right)^{-\alpha(1-\gamma)},
\ j\ne 0 ,
\end{eqnarray}

\begin{eqnarray}\label{sumdelta0k-}\sum\limits_{k\in\mathbb
Z}|\delta_{0k}(x)-\delta_{0k}(y)|\le 2R_0^{1-\gamma}(\alpha)\left(3\Phi^{\gamma}(0)+4\int_{
1/2}^{\infty}\Phi^{\gamma}(t)\,dt\right)\nonumber\\
\times\left(\ln\left(e^{\alpha}+\frac1{|x-y|}\right)\right)^{-\alpha(1-\gamma)}.
\end{eqnarray}
\end{lem}
\begin{proof}
By lemma~\ref{lem2} for $j\ge 1$ we obtain
\begin{eqnarray*}\sum\limits_{k\in\mathbb
Z}|\delta_{jk}(x)-\delta_{jk}(y)|=\sum\limits_{k\in\mathbb
Z}|\delta_{jk}(x)-\delta_{jk}(y)|^{1-\gamma}\cdot|\delta_{jk}(x)-\delta_{jk}(y)|^{\gamma}\\
\le
\left(\ln\left(e^{\alpha}+\frac1{|x-y|}\right)\right)^{-\alpha(1-\gamma)}\left(2^{\frac{j}{2}}
j^{\alpha}R_1(\alpha)\right)^{1-\gamma}\cdot\sum\limits_{k\in\mathbb
Z}|\delta_{jk}(x)-\delta_{jk}(y)|^{\gamma}.
\end{eqnarray*}

We now make use of the inequality $|a+b|^\alpha\le q_\alpha\left(|a|^\alpha+|b|^\alpha\right), $ where
\begin{equation}\label{c_alpha}
q_{\alpha}=\begin{cases} 1,& \alpha\le 1,\\
2^{\alpha-1},& \alpha>1.
\end{cases}\end{equation}
By lemma~\ref{lem1}  we get
\begin{eqnarray*}\sum\limits_{k\in\mathbb
Z}|\delta_{jk}(x)-\delta_{jk}(y)|^{\gamma}&\le & 2\sup\limits_{x\in\mathbb
R}\sum\limits_{k\in\mathbb
Z}|\delta_{jk}(x)|^{\gamma}=2^{\frac{j\gamma}{2}+1}\sup\limits_{x\in\mathbb
R}\sum\limits_{k\in\mathbb
Z}|\delta(2^jx-k)|^{\gamma}\\
&\le &2^{\frac{j\gamma}{2}+1}\left(3\Phi^{\gamma}(0)+4\int_{1/2}^{\infty}\Phi^{\gamma}(t)\,dt\right).
\end{eqnarray*}
Inequality (\ref{sumdeltajk-}) follows from  this estimate.  The proof of inequality~(\ref{sumdelta0k-}) is similar.

\end{proof}
Now we are ready to prove Theorem~\ref{mainuniform}.

\begin{proof}
In virtue of Theorem~\ref{213} and Remark~\ref{772}  it is
sufficient to prove that $\mathbf{X}_{n,\mathbf{k}_n}(t)\to
\mathbf{X}(t)$ in mean square and
$$\left(\mathbf{E}\left|\mathbf{X}_{n,\mathbf{k}_n}(t)-\mathbf{X}_{n,\mathbf{k}_n}(s)\right|^2\right)^{\frac12}\le
\frac
C{\ln^{\beta}\left(e^{\alpha}+\frac1{|t-s|}\right)}$$
for some $C>0, \ \beta>1/2,$ and $\alpha>0.$

Since $\gamma<1,$ assumption $S(1)$ holds true for $\phi(x)$ and
$\psi(x)$. Hence, by Theorem~\ref{8103} we get
$\mathbf{X}_{n,\mathbf{k}_n}(t)\to \mathbf{X}(t)$ in mean square.

Note that (\ref{x_nk}) implies
\begin{eqnarray*}\left(\mathbf{E}\left|\mathbf{X}_{n,\mathbf{k}_n}(t)-
\mathbf{X}_{n,\mathbf{k}_n}(s)\right|^2\right)^{\frac12} \le
\left(\mathbf E\left|\sum\limits_{|k|\le
k_0'}\xi_{0k}(\phi_{0k}(t)-\phi_{0k}(s))\right|^2\right)^{1/2}\\+\sum\limits_{j=0}^{n-1}\left(\mathbf
E\left|\sum\limits_{|k|\le
k_j}\eta_{jk}(\psi_{jk}(t)-\psi_{jk}(s))\right|^2\right)^{1/2}=:\sqrt{S}+\sum\limits_{j=0}^{n-1}\sqrt{S_j}\,.
\end{eqnarray*}
We will only show how to handle $S_j.$ A similar approach can be
used to deal with the remaining term $S.$

$S_j$ can be bounded as follows
$$S_j=\sum\limits_{|k|\le k_j}\sum\limits_{|l|\le k_j}
 \mathbf
 E\eta_{jk}\overline{\eta_{jl}}(\psi_{jk}(t)-\psi_{jk}(s))(\psi_{jl}(t)-\psi_{jl}(s))
\le c_{j}\left(\sum\limits_{k\in\mathbb
Z} |\psi_{jk}(t)-\psi_{jk}(s)|\right)^2. $$
By lemma~\ref{lemsumdelta}
$$S_j\le 2^jj^{2\alpha(1-\gamma)}c_jL^2\left(\ln\left(e^{\alpha}+\frac1{|t-s|}\right)\right)^{-2\alpha(1-\gamma)},$$ where
$$L:=2R_1^{1-\gamma}(\alpha)\left(3\Phi^{\gamma}(0)+4\int_{1/2}^{\infty}\Phi^{\gamma}(t)\,dt\right).$$
Similarly,
$$S\le b_0 L_1^2\left(\ln\left(e^{\alpha}+\frac1{|t-s|}\right)\right)^{-2\alpha(1-\gamma)},$$ where
$$L_1:=2 R_0^{1-\gamma}(\alpha)\left(3\Phi^{\gamma}(0)+4\int_{1/2}^{\infty}\Phi^{\gamma}(t)\,dt\right).$$
Therefore, we conclude that
$$\left(\mathbf{E}\left|\mathbf{X}_{n,\mathbf{k}_n}(t)-\mathbf{X}_{n,\mathbf{k}_n}(s)\right|^2\right)^{\frac12}\le
\frac
C{\left(\ln\left(e^{\alpha}+\frac1{|t-s|}\right)\right)^{\alpha(1-\gamma)}},$$
where
$$C:=\sqrt{b_0}L_1+L\sum\limits_{j=0}^{\infty}\sqrt{c_j}2^{\frac
j2}j^{\alpha(1-\gamma)}<\infty\ \ \mbox{and} \ \ \beta:=\alpha(1-\gamma)>1/2.$$
\end{proof}

\section{Examples}\label{sec5}
In this section we consider some examples of wavelets and stochastic processes which satisfy assumption~(\ref{sum_cj}) of Theorem~\ref{mainuniform}.
\begin{expl}
Let  $\widehat{\psi}$ be a Lipschitz function of order $\varkappa>0$ ($\varkappa=1$ for the case   $\sup_{u\in \mathbb{R}}|\widehat{\psi}'(u)|\le C$), i.e.
$$|\widehat{\psi}(u)-\widehat{\psi}(v)|\le C|u-v|^{\varkappa}.$$ Assume that for the covariance function  $R(t,s)$  $$\int_{\mathbb
R}\int_{\mathbb R}|R(t,s)|\,dt\,ds<\infty$$ and
$$\int_{\mathbb R}\int_{\mathbb
R}\left|\widehat{R}_2(z,w)\right|\cdot\left|z\right|^{\varkappa}\cdot\left|w\right|^{\varkappa}\,dz\,dw<\infty,$$
where $$\widehat{R}_2(z,w)=\int_{\mathbb
R}\int_{\mathbb R}{R}(u,v)\,e^{-izu}e^{-iwv}\,du\,dv.$$

Now we show that $\mathbf E|\eta_{jk}|^2\le c_j$ for all  $k\in\mathbb Z$ and find suitable upper bounds for $c_j.$

By Parseval's theorem,
$$\mathbf
E\eta_{jk}\overline{\eta_{jl}}=\int_{\mathbb
R}\int_{\mathbb R}\mathbf E
\mathbf{X}(u)\overline{\mathbf{X}(v)}\
\overline{\psi_{jk}(u)}\psi_{jl}(v)\,dudv$$$$=\frac{1}{(2
\pi)^2}\int_{\mathbb R}\int_{\mathbb
R}\widehat{R}_2(z,w)\,\overline{\widehat{\psi}_{jk}(z)}\,\widehat{\psi}_{jl}(w)\,dz\,dw.$$
Since
\[\label{hatpsijk}\widehat{\psi}_{jk}(z)=\frac{e^{-i\frac k{2^j}z}}{2^{j/2}}\cdot\widehat{\psi}\left(\frac
z{2^j}\right),
\]
it follows that
\[\left|\mathbf E\eta_{jk}\overline{\eta_{jl}}\right|\le\frac{1}{2^{j}(2
\pi)^2}\int_{\mathbb R}\int_{\mathbb
R}\left|\widehat{R}_2(z,w)\right|\cdot\left|\widehat{\psi}\left(\frac
z{2^j}\right)\right|\cdot\left|\widehat{\psi}\left(\frac
w{2^j}\right)\right|\,dz\,dw.\]

By properties of the $m$-wavelet $\psi$
we have $\widehat{\psi}(0)=0.$ Therefore, using the Lipschitz conditions, we obtain
$$\left|\mathbf
E\eta_{jk}\overline{\eta_{jl}}\right|\le\frac{C^2}{(2
\pi)^2 2^{j(1+2\varkappa)}}\int_{\mathbb
R}\int_{\mathbb
R}\left|\widehat{R}_2(z,w)\right|\cdot\left|z\right|^{\varkappa}\cdot\left|w\right|^{\varkappa}\,dz\,dw.$$
This means that $\sqrt{c_j}\le C/{2^{j/2(1+2\varkappa)}}$ and assumption~(\ref{sum_cj}) holds.
\end{expl}
In the following example we consider the case of stationary stochastic processes. This case was studied in detail by us in \cite{kozol1}. Note that
 assumptions in the example are much simpler than those used in \cite{kozol1}.
\begin{expl}
Let  $\mathbf{X}(t)$ be a centered  short-memory stationary stochastic process and $\widehat{\psi}$ be a Lipschitz function of order $\varkappa>0.$ Assume that the covariance function $R(t-s):=\mathbf E\mathbf{X}(t)\overline{\mathbf{X}(s)}$ satisfies the following condition
$$\int_{\mathbb R}\left|\widehat{R}(z)\right|\cdot\left|z\right|^{2\varkappa}\,dz<\infty.$$

By  Parseval's theorem we deduce
$$|\mathbf E\eta_{jk}\overline{\eta_{jl}}|=\left|\int_{\mathbb R}\int_{\mathbb R}R(u-v)\overline{\psi_{jk}(u)}\,du\,\psi_{jl}(v)\,dv\right|=\left|\int_{\mathbb R}\int_{\mathbb R}\frac{e^{-ivz}}{2 \pi}\widehat{R}(z)\overline{\widehat{\psi}_{jk}(z)}\,dz\psi_{jl}(v)\,dv\right|$$
$$\le\frac{1}{2 \pi}\int_{\mathbb R}\left|\widehat{R}(z)\right|\,\left|\overline{\widehat{\psi}_{jk}(z)}\,\widehat{\psi}_{jl}(z)\right|\,dz=\frac{1}{2^{j+1}
\pi}\int_{\mathbb
R}\left|\widehat{R}(z)\right|\cdot\left|\widehat{\psi}\left(\frac
z{2^j}\right)\right|\cdot\left|\widehat{\psi}\left(\frac
z{2^j}\right)\right|\,dz.$$
Thus, by the Lipschitz conditions, for all $k,l\in\mathbb Z:$
$$\left|\mathbf
E\eta_{jk}\overline{\eta_{jl}}\right|\le\frac C{\pi\, 2^{1+j(1+2\varkappa)}}\int_{\mathbb
R}\left|\widehat{R}(z)\right|\cdot\left|z\right|^{2\varkappa}\,dz.$$
This means that $\sqrt{c_j}\le C/2^{j/2(1+2\varkappa)}$ and assumption~(\ref{sum_cj}) is satisfied.
\end{expl}

\section{Application to fractional Brownian motion}
In this section we show how to adjust the main theorem to the fractional Brownian motion case.

Let $W_{\alpha}(t),$ $t\in\mathbb R,$ be a separable centered Gaussian random process such that $W_{\alpha}(-t)=W_{\alpha}(t)$ and its covariance function is
\begin{equation}\label{cov}
R(t,s)=\mathbf E
W_{\alpha}(t)W_{\alpha}(s)=
\frac12\left(|t|^{\alpha}+|s|^{\alpha}-||t|-|s||^{\alpha}\right),\ 0<{\alpha}<2.
\end{equation}

\begin{lem}\label{lem4}
If assumption $S(\gamma),$ $0<\gamma<1,$ holds true and for some $\alpha>0$
$$c_\psi:=\int_{\mathbb R}|u|^{\alpha}|{\psi}(u)|\,du<\infty,$$
then for the coefficients of the process $W_{\alpha}(t),$ defined by {\rm(\ref{intksieta})},
\begin{eqnarray*}\left|\mathbf E\eta_{jk}\overline{\eta_{jl}}\right|\le \frac C{2^{j(1+\alpha)}}\end{eqnarray*}
for all $k,l\in \mathbb Z.$
\end{lem}

\begin{proof}
Since  $\left|\mathbf E\eta_{jk}\overline{\eta_{jl}}\right|\le\left(\mathbf
E|\eta_{jk}|^2\right)^{\frac12}\left(\mathbf
E|\eta_{jl}|^2\right)^{\frac12},$ it is sufficient to estimate $\mathbf
E|\eta_{jk}|^2.$

By (\ref{psi_jk}) and (\ref{cov}) we obtain
\begin{eqnarray}\label{E|eta_jk|^2}\mathbf
E|\eta_{jk}|^2&=&\int_{\mathbb R}\int_{\mathbb
R}R(t,s)\
\overline{\psi_{jk}(t)}\psi_{jk}(s)\,dt ds =\frac 1{2^j}\int_{\mathbb R}\int_{\mathbb
R}R\left(\frac u{2^j}, \frac v{2^j}\right)\overline{\psi(u-k)}\psi(v-k)\,du\,dv\nonumber\\
&=&\frac 1{2^j}\int_{\mathbb R}\int_{\mathbb
R}\frac12\left(\left|\frac u{2^j}\right|^{\alpha}+\left|\frac
v{2^j}\right|^{\alpha}-\left|\left|\frac u{2^j}\right|-\left|\frac
v{2^j}\right|\right|^{\alpha}\right)\overline{\psi(u-k)}\psi(v-k)\,du\,dv\nonumber\\
&\le&\frac 1{2^{1+j(1+\alpha)}}\left(\int_{\mathbb
R}\int_{\mathbb R}\left|
u\right|^{\alpha}\overline{\psi(u-k)}\psi(v-k)\,du\,dv\right.\nonumber\\&+&\int_{\mathbb
R}\int_{\mathbb R}\left|
v\right|^{\alpha}\overline{\psi(u-k)}\psi(v-k)\,du\,dv\nonumber\\&+&\left.\int_{\mathbb
R}\int_{\mathbb R}\left|\left| u\right|-\left|
v\right|\right|^{\alpha}\overline{\psi(u-k)}\psi(v-k)\,du\,dv\right)=:\frac {z_1+z_2+z_3}{2^{1+j(1+\alpha)}}.\end{eqnarray}
It follows from assumption $S(\gamma)$  that
$$N:=\int_{\mathbb R}|{\psi}(u)|\,du<\infty.$$
Hence, we get the estimate
\begin{eqnarray}\label{z_1} z_1&\le& \int_{\mathbb R}\int_{\mathbb
R}\left|
u\right|^{\alpha}|{\psi(u-k)}||\psi(v-k)|\,du\,dv=\int_{\mathbb
R}|\psi(v)|\,dv \nonumber\\
&\times&\int_{\mathbb
R}\left| u+k\right|^{\alpha}|{\psi(u)}|\,du\le
q_{\alpha}{N}\left(N|k|^{\alpha}+c_{\psi}\right)<\infty,\end{eqnarray}
where $q_{\alpha}$ is defined by (\ref{c_alpha}).

Using Fubini's theorem and $\int_{\mathbb R} \psi(u)\,du=0$ we obtain
\begin{eqnarray*}z_1=\int_{\mathbb R}\left| u\right|^{\alpha}\overline{\psi(u-k)}\,du\int_{\mathbb
R}\psi(v-k)\,dv=0.
\end{eqnarray*}
Similarly, $z_2=0.$

Finally, we estimate $z_3$ as follows
\[z_3\le\int_{\mathbb
R}\int_{\mathbb R}\left|\left| u+k\right|-\left|
v+k\right|\right|^{\alpha}|{\psi(u)}||\psi(v)|\,du\,dv.
\]

By the reverse triangle inequality and (\ref{c_alpha}) we obtain
$$\left|\left| u+k\right|-\left|
v+k\right|\right|^{\alpha}\le  q_{\alpha}\left(\left|
u\right|^{\alpha}+\left| v\right|^{\alpha}\right).$$

Hence,
\begin{eqnarray}\label{z_3}|z_3|&\le&q_{\alpha} \ 2^{\alpha+1}\int_{\mathbb
R}\left| u\right|^{\alpha}|{\psi(u)}|\,du\int_{\mathbb
R}|\psi(v)|\,dv=q_{\alpha} c_{\psi}\ 2^{\alpha+1}
N,
\end{eqnarray}
which completes the proof of the lemma.

\end{proof}

In some case, for example, for the fractional Brownian motion
the assumption $|\mathbf E\xi_{0k}\overline{\xi_{0l}}|\le b_0$ of
Theorem~\ref{mainuniform} doesn't hold true. The following
theorem gives the uniform convergence of wavelet expansions without this assumption.

\begin{thm}\label{th5}
Let  a random process $\mathbf{X}(t),$ $t\in\mathbb R,$  the $f$-wavelet $\phi,$
and the corresponding $m$-wavelet $\psi$  satisfy the assumptions
of Theorem~{\rm\ref{8103}} and assumptions {\rm(i)} and {\rm(ii)} of Theorem~{\rm\ref{mainuniform}}.

Suppose that there exist
\begin{enumerate}
  \item[{\rm(iii')}]
constants  $c_j,$ $j\in \mathbb N_0,$ such that $\mathbf
E|\eta_{jk}|^2\le c_j$  for all $k\in\mathbb Z$ and
{\rm(\ref{sum_cj})} holds true;
\item[{\rm(iv)}]  some $\varepsilon>0$ such that
 $$S:=\mathbf{E}\left|\sum\limits_{|k|<|k_0'|}\xi_{0k}(\phi_{0k}(t)-\phi_{0k}(s))\right|^2
 \le
\frac
C{\left(\ln\left(e^{\alpha}+\frac1{|t-s|}\right)\right)^{2\alpha(1-\gamma)}},\
\ \alpha>1/2(1-\gamma),$$
if $|t-s|<\varepsilon.$
\end{enumerate}
Then $\mathbf{X}_{n,\mathbf{k}_n}(t)\to \mathbf{X}(t)$
uniformly in probability on each interval $[0,T]$ when  $n\to\infty,$
$k_0'\to \infty,$  and $k_j\to\infty$ for all $j\in \mathbb N_0\,.$
\end{thm}

\begin{proof}
 The assertion of the theorem follows from the proof of Theorem~\ref{mainuniform}.
\end{proof}

Now, under some mild additional conditions on  the $f$-wavelet $\phi,$ we show that estimate (iv) holds true in the fractional Brownian
 motion case.

\begin{lem}\label{lem}
Let   the $f$-wavelet $\phi$ satisfy the assumptions
of Theorem~{\rm\ref{8103}},
\begin{itemize}
\item $\widehat{\phi}(z)\to 0$ and $\widehat{\phi}'(z)\to 0,$ when $z\to\pm\infty;$
\item the  integrals $\int_{\mathbb
R}|u|^{\alpha}|\phi(u)|,$
$\int_{\mathbb
R}|\widehat{\phi}'(u)|\,du$ and
$\int_{\mathbb
R}\ln^{\alpha}(1+|u|)|\widehat{\phi}^{(i)}(u)|\,du,$ $i=0,1,2,$ converge for some
$\alpha>1/{2(1-\gamma)}.$
\end{itemize}
Then assumption {\rm(iv)} of Theorem~{\rm\ref{th5}} holds true for $W_{\alpha}(t).$
\end{lem}
\begin{proof}
$S$ can be bounded as follows
 $$S\le \sum\limits_{|k|\le k_0}\sum\limits_{|l|\le k_0}|\mathbf E\xi_{0k}\overline{\xi_{0l}}|
 |\phi_{0k}(t)-\phi_{0k}(s)||\phi_{0l}(t)-\phi_{0l}(s)|\,.$$
Since $\left|\mathbf
E\xi_{0k}\overline{\xi_{0l}}\right|\le\left(\mathbf
E|\xi_{0k}|^2\right)^{\frac12}\left(\mathbf
E|\xi_{0l}|^2\right)^{\frac12},$ it is sufficient to estimate $\mathbf
E|\xi_{0k}|^2.$

Similarly to (\ref{E|eta_jk|^2}) we get
 \begin{eqnarray*}\mathbf
E|\xi_{0k}|^2&\le&\int_{\mathbb
R}\int_{\mathbb R}\left|
u\right|^{\alpha}\overline{\phi(u-k)}\phi(v-k)\,du\,dv+\int_{\mathbb
R}\int_{\mathbb R}\left|
v\right|^{\alpha}\overline{\phi(u-k)}\phi(v-k)\,du\,dv\nonumber\\&+&\int_{\mathbb
R}\int_{\mathbb R}\left|\left| u\right|-\left|
v\right|\right|^{\alpha}\overline{\phi(u-k)}\phi(v-k)\,du\,dv.
\end{eqnarray*}
Analogously to (\ref{z_1}) and (\ref{z_3}) we obtain
$$\int_{\mathbb
R}\int_{\mathbb R}\left|
u\right|^{\alpha}|\overline{\phi(u-k)}\phi(v-k)|\,du\,dv\le
q_{\alpha}{M}\left(M|k|^{\alpha}+c_{\phi}\right),$$
$$\int_{\mathbb
R}\int_{\mathbb R}\left|\left| u\right|-\left|
v\right|\right|^{\alpha}|\overline{\phi(u-k)}\phi(v-k)|\,du\,dv\le q_{\alpha}c_{\phi}M \, 2^{\alpha+1},
$$
where
$$M:=\int_{\mathbb
R} |\phi(u)|\,du<\infty, \ c_{\phi}:=\int_{\mathbb
R}|u|^{\alpha} |\phi(u)|\,du<\infty.$$

Then
\begin{eqnarray}\label{Exi0kl}\left|\mathbf
E\xi_{0k}\overline{\xi_{0l}}\right|&&\le 2q_{\alpha}\
{M}\left(M|k|^{\alpha}+c_{\phi}(2^{\alpha}+1)\right)^{\frac12}
\left(M|l|^{\alpha}+c_{\phi}(2^{\alpha}+1)\right)^{\frac12}\nonumber\\
&&\le 2q_{\alpha}\
{M}\left(M|k|^{\frac{\alpha}2}+\sqrt{c_{\phi}2^{\alpha+1}}\right)
\left(M|l|^{\frac{\alpha}2}+\sqrt{c_{\phi}2^{\alpha+1}}\right).\nonumber\end{eqnarray}

To estimate  $|{\phi}_{0k}(t)-{\phi}_{0k}(s)|,$ we
use the representations
$$\widehat{\phi}_{0k}(z)={e^{-i
kz}}\cdot\widehat{\phi}\left( z\right),
$$
$${\phi}_{0k}(t)=\frac1{2\pi}\int_{\mathbb R}
 {e^{itz}e^{-ikz}}\widehat{\phi}\left(z\right)\,dz.$$
Repeatedly using integration by parts and
the assumptions of the lemma, we obtain that for $k\not= 0:$
\begin{eqnarray}\label{phi0kt-phi0ks} &&|{\phi}_{0k}(t)-{\phi}_{0k}(s)|=\left|\frac1{2i
 k\pi}\int_{\mathbb R}
 \left(e^{itz}-e^{isz}\right)\widehat{\phi}\left(z\right)\,d(e^{-ikz})\right|\nonumber\\
 &&=\frac1{2\pi k^2}\left|\int_{\mathbb R}\left[
 i\left(te^{itz}-se^{isz}\right)\widehat{\phi}(z)+\left(e^{itz}-e^{isz}\right)\widehat{\phi}'(z)\right]\,d(e^{-ikz})\right|
\nonumber \\
 &&=\frac1{2\pi k^2}\left|\int_{\mathbb R}\left[
 -\left(t^2e^{itz}-s^2e^{isz}\right)\widehat{\phi}(z)+2i\left(te^{itz}-se^{isz}\right)
 \widehat{\phi}'(z)\right.\right.\nonumber\\&&+\left.
 \left(e^{itz}-e^{isz}\right)\widehat{\phi}''(z)\right]e^{-ikz}\,dz\left|
 \le\frac1{2\pi k^2}\right.\left(\int_{\mathbb R}
 \left|t^2e^{itz}-s^2e^{isz}\right||\widehat{\phi}(z)|\,dz\right.\nonumber\\&&+\left.2\int_{\mathbb
 R}\left|te^{itz}-se^{isz}\right|
| \widehat{\phi}'(z)|\,dz+\int_{\mathbb R}
 \left|e^{itz}-e^{isz}\right||\widehat{\phi}''(z)|\,dz\right).
\end{eqnarray}

By inequalities (8) and (12) given in \cite{kozol1} we get
\begin{eqnarray*}\label{te-se}|te^{itz}-se^{isz}|&\le&|t-s|+t|e^{itz}-e^{isz}|
\\&\le&\frac{c_{\alpha,T}}{\ln^{\alpha}\left(e^{\alpha}+
\frac1{|t-s|}\right)}+2T\left(\frac{\ln\left(e^{\alpha}+
\frac{|z|}2\right)}{\ln\left(e^{\alpha}+
\frac1{|t-s|}\right)}\right)^{\alpha},
\end{eqnarray*}
\begin{eqnarray*}\label{t2e-s2e}|t^2e^{itz}-s^2e^{isz}|&\le&|t^2-s^2|+t^2|e^{itz}-e^{isz}|
\\&\le&\frac{\tilde{c}_{\alpha,T}}{\ln^{\alpha}\left(e^{\alpha}+
\frac1{|t-s|}\right)}+2T^2\left(\frac{\ln\left(e^{\alpha}+
\frac{|z|}2\right)}{\ln\left(e^{\alpha}+
\frac1{|t-s|}\right)}\right)^{\alpha},
\end{eqnarray*}
where $c_{\alpha,T}$ and $\tilde{c}_{\alpha,T}$ are constants which do not depend on $t,$ $s$ and $z.$

Applying these inequalities to (\ref{phi0kt-phi0ks}) we obtain
\[|{\phi}_{0k}(t)-{\phi}_{0k}(s)|
 \le\frac{B_{\phi,\alpha,T}}{k^2\ln^{\alpha}\left(e^{\alpha}+
\frac1{|t-s|}\right)},\quad k\not= 0,
\]
where
\begin{eqnarray*}B_{\phi,\alpha,T}&:=&\frac1{2\pi}\left(\tilde{c}_{\alpha,T}\int_{\mathbb R}
 |\widehat{\phi}(z)|\,dz+2T^2\int_{\mathbb R}\ln^{\alpha}\left(e^{\alpha}+
\frac{|z|}2\right)
 |\widehat{\phi}(z)|\,dz\right.\nonumber\\&+&2c_{\alpha,T}\int_{\mathbb R}
 |\widehat{\phi}'(z)|\,dz+4T^2\int_{\mathbb R}\ln^{\alpha}\left(e^{\alpha}+
\frac{|z|}2\right)
 |\widehat{\phi}'(z)|\,dz\nonumber\\ &+&\left.2\int_{\mathbb R}\ln^{\alpha}\left(e^{\alpha}+
\frac{|z|}2\right)
 |\widehat{\phi}''(z)|\,dz\right).
\end{eqnarray*}

If $k=0,$ then
\[|{\phi}_{00}(t)-{\phi}_{00}(s)|\le\frac1{2
 \pi}\int_{\mathbb R}
 \left|e^{itz}-e^{isz}\right|\,|\widehat{\phi}\left(z\right)|\,dz\le\frac{B_{0,\alpha,T}}{\ln^{\alpha}\left(e^{\alpha}+
 \frac1{|t-s|}\right)},\]
 where
 \[B_{0,\alpha,T}:=\frac1{\pi}\int_{\mathbb R}\ln^{\alpha}\left(e^{\alpha}+
 \frac{|z|}2\right)
  |\widehat{\phi}(z)|\,dz\,.
 \]
Consequently, we can estimate $S$ as follows
$$S \le
\frac
C{\left(\ln\left(e^{\alpha}+\frac1{|t-s|}\right)\right)^{2\alpha(1-\gamma)}},\
\ \alpha>1/2(1-\gamma),$$ where
\[C:=2q_{\alpha} M \Bigg(\sqrt{c_{\phi}2^{\alpha+1}}\,B_{0,\alpha,T}+B_{\phi,\alpha,T}
\sum\limits_{\scriptsize \begin{array}{c}
|k|\le k_0\\
k\not=0
\end{array}}\frac{M|k|^{\frac{\alpha}2}+\sqrt{c_{\phi}2^{\alpha+1}}}{k^2}\Bigg)^2.
\]
\end{proof}
\begin{thm}
If the assumptions of  Lemmata~{\rm\ref{lem4}, \ref{lem}}, and assumptions {\rm(i)} and {\rm(ii)} of Theorem~{\rm\ref{mainuniform}} are satisfied, then the wavelet expansions of the fractional Brownian motion uniformly converge to $W_{\alpha}(t).$
\end{thm}

\section{Convergence rate in the space $C[0, T]$}
Returning now to the general case introduced in Theorem~\ref{mainuniform}, let us investigate what happens when the number of terms in the approximants~(\ref{x_nk}) becomes large.

First we specify an estimate for the supremum of Gaussian processes.

\begin{defn}{\rm\cite[\S3.2]{bulkoz}} A set $\mathcal{Q}\subset \mathcal{S} \subset \mathbb{R}$ is called an $\varepsilon$-net in the set $\mathcal{S}$ with respect to the semimetric $\rho$ if for any point $x\in \mathcal{S}$ there exists at least one point $y\in \mathcal{Q}$ such that $\rho(x,y)\le \varepsilon.$
\end{defn}

\begin{defn}{\rm\cite[\S3.2]{bulkoz}} Let
$$H_\rho(\mathcal{S},\varepsilon):=\left\{
\begin{array}{ll}
\ln(N_\rho(\mathcal{S},\varepsilon)), & \hbox{if}\ N_\rho(\mathcal{S},\varepsilon)<+\infty; \\
+\infty, & \hbox{if}\ N_\rho(\mathcal{S},\varepsilon)=+\infty,
\end{array}
\right.
$$
where $N_\rho(\mathcal{S},\varepsilon)$ is the number of point in a
minimal $\varepsilon$-net in the set $\mathcal{S}.$

The function $H_\rho(\mathcal{S},\varepsilon),$ $\varepsilon>0,$ is
called the metric entropy of the set $\mathcal{S}.$
\end{defn}

\begin{lem}\label{71}{\rm\cite[(4.10)]{bulkoz}} Let $\mathbf
Y(t),$ $t\in[0,T]$ be a separable Gaussian random process,
$$\varepsilon_0:=\sup_{0\le t\le T}\left(\mathbf
\mathbf
E|\mathbf
Y(t)|^2\right)^{1/2}<\infty\,,$$
\begin{equation}\label{76}I(\varepsilon_0):=\frac1{\sqrt2}\int\limits_{0}^{\varepsilon_0}
\sqrt{H(\varepsilon)}\,d\varepsilon<\infty\,,\end{equation}
where $H(\varepsilon)$ is the metric entropy of the space $([0,
T],\rho),$ $\rho(t,s)=(\mathbf
E|\mathbf
Y(t)-\mathbf
Y(s)|^2)^{1/2}.$

Then
$$P\left\{\sup_{0\le t\le T} |\mathbf
Y(t)|>u \right\}\le
2\exp\left\{-\frac{\left(u-\sqrt{8uI(\varepsilon_0)}\right)^2}{2\varepsilon_0^2}\right\}\,,$$
where $u>8I(\varepsilon_0).$
\end{lem}

Assume that there exists a nonnegative monotone nondecreasing in some neighborhood of the origin
function $\sigma(\varepsilon),$ $\varepsilon>0,$ such that
$\sigma(\varepsilon)\to 0$ when $\varepsilon\to 0$ and
\begin{equation}\label{74}\sup\limits_{\scriptsize\begin{array}{c}
                     |t-s|\le \varepsilon\\
                      t,s\in [0, T]
                   \end{array}}\left(\mathbf
                   E|\mathbf
Y(t)-\mathbf
Y(s)|^2\right)^{1/2}\le\sigma(\varepsilon)\,.
\end{equation}
\begin{lem}\label{lem7}{\rm\cite{kozol2}} If
\begin{equation}\label{sigma}\sigma(\varepsilon)=\frac C{\ln^{\beta}\left(e^{\alpha}+
\frac1{\varepsilon}\right)},  \
\beta>1/2,\ \alpha>0\,,\end{equation} then {\rm(\ref{76})} holds true and
$$I(\varepsilon_0)\le \delta(\varepsilon_0):=\frac{\gamma}{\sqrt2}\left(\sqrt{\ln(T+1)}+\left({1-\frac{1}{2\beta}}\right)^{-1}\left(\frac{C}{\gamma}\right)^{\frac{1}{2\beta}}\right),$$
where
$\gamma:=\min\left(\varepsilon_0,\,\sigma\left(\frac{T}{2}\right)\right)\,.$
\end{lem}

\begin{lem}\label{lemSgamma2} If a scaling function $\delta(x)$ satisfies assumption $S'(\gamma),$ then
\[\sup\limits_{|x|\le T}\sum\limits_{|k|\ge k_1}|\delta(x-k)|^{\gamma}
\le\int\limits_{k_1-T-1}^{\infty}\Phi^{\gamma}(t)\,dt+\int\limits_{k_1-1}^{\infty}\Phi^{\gamma}(t)\,dt
\]
for  $k_1\ge T+1.$
\end{lem}
\begin{proof}
Since
$$\sum\limits_{|k|\ge k_1}|\delta(x-k)|^{\gamma}\le\sum\limits_{k\ge
k_1}\left(\Phi^{\gamma}(|x+k|)
+\Phi^{\gamma}(|x-k|)\right)=:z_{k_1}(x),$$
$z_{k_1}(x)$ is an even function.

Then the assertion of the theorem follows from
\begin{eqnarray*}\sup\limits_{|x|\le
T}\sum\limits_{|k|\ge k_1}|\delta(x-k)|^{\gamma}\le\sup\limits_{0\le
x\le T}z_{k_1}(x)\le\sum\limits_{k\ge
k_1}\left(\Phi^{\gamma}(k-T)
+\Phi^{\gamma}(k)\right)\\
\le\sum\limits_{k\ge
k_1}\left(\,\int\limits_{k-1}^{k}\Phi^{\gamma}(t-T)\,dt+
\int\limits_{k-1}^{k}\Phi^{\gamma}(t)\,dt\right)\le\int\limits_{k_1-T-1}^{\infty}
\Phi^{\gamma}(t)\,dt+\int\limits_{k_1-1}^{\infty}\Phi^{\gamma}(t)\,dt.
\end{eqnarray*}
\end{proof}

Now we formulate the main result of this section.

\begin{thm}\label{th7} Let  a separable Gaussian random process $\mathbf{X}(t),$ $t\in [0,T],$  the $f$-wavelet $\phi,$
and the corresponding $m$-wavelet $\psi$  satisfy the assumptions of Theorem~{\rm\ref{mainuniform}}.

Then
$$P\left\{\sup_{0\le t\le T} |\mathbf{X}(t)-\mathbf{X}_{n,\mathbf{k}_n}(t)|>u \right\}\le
2\exp\left\{-\frac{(u-\sqrt{8u\delta(\varepsilon_{\mathbf{k}_n})})^2}{2\varepsilon_{\mathbf{k}_n}^2}\right\}\,,$$
where $u>8\delta(\varepsilon_{\mathbf{k}_n})$ and
the decreasing sequence $\varepsilon_{\mathbf{k}_n}$ is defined by {\rm(\ref{varepsilon_k_n})} in the proof of the theorem.
\end{thm}
\begin{proof}
Let us verify that
$\mathbf
Y(t):=\mathbf{X}(t)-\mathbf{X}_{n,\mathbf{k}_n}(t)$ satisfies
(\ref{74}) with $\sigma(\varepsilon)$ given by
(\ref{sigma}).

First, we observe that
$$\left(\mathbf
E|\mathbf
Y(t)-\mathbf
Y(s)|^2\right)^{1/2}=\left(\mathbf
E\left|\sum_{|k|>k_0'}\xi_{0k}(\phi_{0k}(t)-\phi_{0k}(s))\right.\right.$$
$$\left.\left.+\sum_{j=0}^{n-1}\sum_{|k|>k_j}\eta_{jk}(\psi_{jk}(t)-\psi_{jk}(s))+\sum_{j=n}^{\infty}\sum_{k\in\mathbb
Z}\eta_{jk}(\psi_{jk}(t)-\psi_{jk}(s))\right| ^2\right)^{1/2}$$
$$\le\left(\mathbf
E\left|\sum_{|k|>k_0'}\xi_{0k}(\phi_{0k}(t)-\phi_{0k}(s))\right|
^2\right)^{1/2}+
\sum_{j=0}^{n-1}\left(\mathbf
E\left|\sum_{|k|>k_j}\eta_{jk}(\psi_{jk}(t)-\psi_{jk}(s))\right|
^2\right)^{1/2}$$
$$+\sum_{j=n}^{\infty}\left(\mathbf
E\left|\sum_{k\in\mathbb Z}\eta_{jk}(\psi_{jk}(t)-\psi_{jk}(s))\right| ^2\right)^{1/2}
:=\sqrt{S'}+\sum_{j=0}^{n-1}\sqrt{S'_j}+\sum_{j=n}^\infty
\sqrt{R'_j}\,.$$

We will only show how to handle $S'_j.$ A similar approach can be
used to deal with the remaining  terms $S'$ and $R'_j.$

By Lemmata~\ref{lem1} and \ref{lem2} we get
 \begin{eqnarray*}S'_j&\le& \sum\limits_{|k|> k_j}\sum\limits_{|l|> k_j}
 |\mathbf E\eta_{jk}\overline{\eta_{jl}}||\psi_{jk}(t)-\psi_{jk}(s)||\psi_{jl}(t)-\psi_{jl}(s)|\\
 &\le& c_j\left(\sum\limits_{|k|> k_j}
|\psi_{jk}(t)-\psi_{jk}(s)|^{\gamma}|\psi_{jk}(t)-\psi_{jk}(s)|^{1-\gamma}\right)^2\\
&\le& c_j\frac{\left(2^{
j/2}j^{\alpha}R_{1\alpha}\right)^{2(1-\gamma)}}{\left(\ln\left(e^{\alpha}+\frac1{|t-s|}
\right)\right)^{2\alpha(1-\gamma)}}\left(\sum\limits_{|k|> k_j}
|\psi_{jk}(t)-\psi_{jk}(s)|^{\gamma}\right)^2\\ &\le&
c_j\frac{\left(2^{
j/2}j^{\alpha}R_{1\alpha}\right)^{2(1-\gamma)}}{\left(\ln\left(e^{\alpha}+\frac1{|t-s|}
\right)\right)^{2\alpha(1-\gamma)}}\left(2\sup\limits_{|t|\le
T}\sum\limits_{|k|> k_j} |\psi_{jk}(t)|^{\gamma}\right)^2\\ &\le&
\frac{c_j2^{
j+2}\left(j^{\alpha}R_{1\alpha}\right)^{2(1-\gamma)}}{\left(\ln\left(e^{\alpha}+\frac1{|t-s|}
\right)\right)^{2\alpha(1-\gamma)}}\left(3\Phi^{\gamma}(0)+4\int_{1/2}^{\infty}\Phi^{\gamma}(t)\,dt\right)^2,
\end{eqnarray*}
where $R_{1\alpha}=\frac1{\pi}\int_{\mathbb
R}\ln^{\alpha}\left(e^{\alpha}+|u|+1\right)|\widehat{\psi}(u)|\,du.$

Hence
\begin{eqnarray}\label{sumSj}\sum\limits_{j=0}^{n-1}\sqrt{S'_j}
\le\frac{B_{\alpha}}{\left(\ln\left(e^{\alpha}+\frac1{|t-s|}
\right)\right)^{\alpha(1-\gamma)}},
\end{eqnarray} where
$$B_{\alpha}:=\sum\limits_{j=0}^{\infty}{\sqrt{c_j}2^{
(j+2)/2}\left(j^{\alpha}R_{1\alpha}\right)^{1-\gamma}}\left(3\Phi^{\gamma}(0)+4\int_{1/2}^{\infty}\Phi^{\gamma}(t)\,dt\right)<\infty\,.$$

From Lemma~\ref{lemsumdelta}, it follows that
\[\sum_{j=n}^\infty \sqrt{R'_j}\le  \frac{L
\cdot\sum_{j=n}^\infty\sqrt{c_j}2^{\frac
{j}2}j^{\alpha(1-\gamma)}} {\left(\ln\left(e^{\alpha}+
\frac1{|t-s|}\right)\right)^{\alpha(1-\gamma)}} ,\]
where
$$L:=2R_1^{1-\gamma}(\alpha) \left(3\Phi^{\gamma}(0)+4\int_{1/2}^{\infty}\Phi^{\gamma}(t)\,dt\right)\,.$$

Similarly to (\ref{sumSj}) by lemma~\ref{lem7}  we obtain
\begin{eqnarray*}\sqrt{S'}
\le\frac{b_0R_{0\alpha}^{1-\gamma}}{\left(\ln\left(e^{\alpha}+\frac1{|t-s|}
\right)\right)^{\alpha(1-\gamma)}}\left(\int_{k'_0-T-1}^{\infty}
\Phi^{\gamma}(t)\,dt+\int_{k'_0-1}^{\infty}\Phi^{\gamma}(t)\,dt\right),
\end{eqnarray*}
where $R_{0\alpha}=\frac1{\pi}\int_{\mathbb
R}\ln^{\alpha}\left(e^{\alpha}+\frac{|u|}2\right)|\widehat{\phi}(u)|\,du$ and $k_0'\ge T+1.$

Thus,
\[\left(\mathbf
E|\mathbf
Y(t)-\mathbf
Y(s)|^2\right)^{1/2}=\left(\mathbf
E\left|(\mathbf{X}(t)-\mathbf{X}_{n,\mathbf{k}_n}(t))-(\mathbf{X}(s)-\mathbf{X}_{n,\mathbf{k}_n}(s))\right|^2\right)^{1/2}\]
\[ \le \frac{C}{\left(\ln\left(e^{\alpha}+
\frac1{|t-s|}\right)\right)^{\alpha(1-\gamma)}}=:\sigma(|t-s|), \ \
\alpha(1-\gamma)>1/2,
\]
where
\[C:=L
\cdot\sum_{j=n}^\infty\sqrt{c_j}2^{\frac
{j}2}j^{\alpha(1-\gamma)}+B_{\alpha}
+b_0 R_{0\alpha}^{1-\gamma}\left(\int_{k'_0-T-1}^{\infty}
\Phi^{\gamma}(t)\,dt+\int_{k'_0-1}^{\infty}\Phi^{\gamma}(t)\,dt\right).\]
Then by lemmata~\ref{71} and \ref{lem7}
$$P\left\{\sup_{0\le t\le T} |\mathbf{X}(t)-\mathbf{X}_{n,\mathbf{k}_n}(t)|>u \right\}\le
2\exp\left\{-\frac{(u-\sqrt{8uI(\tilde{\varepsilon}_{\mathbf{k}_n})})^2}{2\tilde{\varepsilon}_{\mathbf{k}_n}^2}\right\}\,,$$
where $$\tilde{\varepsilon}_{\mathbf{k}_n}=\sup_{0\le t\le T}
\left(\mathbf
E\left|\mathbf{X}(t)-\mathbf{X}_{n,\mathbf{k}_n}(t)\right|^2\right)^{1/2}\,.$$

The proof will be completed by investigating $\tilde{\varepsilon}_{\mathbf{k}_n}$ as a
function of $\mathbf{k}_n.$

Note that
$$\left(\mathbf
E\left|\mathbf{X}(t)-\mathbf{X}_{n,\mathbf{k}_n}(t)\right|^2\right)^{1/2}\le\left(\mathbf
E\left|\sum_{|k|>k_0'}\xi_{0k}\phi_{0k}(t)\right| ^2\right)^{1/2}+$$
$$+\sum_{j=0}^{n-1}\left(\mathbf
E\left|\sum_{|k|>k_j}\eta_{jk}\psi_{jk}(t)\right| ^2\right)^{1/2}+\sum_{j=n}^{\infty}\left(\mathbf
E\left|\sum_{k\in\mathbb Z}\eta_{jk}\psi_{jk}(t)\right| ^2\right)^{1/2}\,.$$

Let $J:=\min \left\{n,\min\{j\in \mathbb N_0:k_j<2^jT+1\}\right\}.$ Notice that $J\to \infty,$ when $n\to\infty$
 and $k_j\to\infty$ for all $j\in \mathbb N_0\,.$

By Lemmata ~\ref{lem2} and \ref{lemSgamma2} for $k_0'\ge T+1$
\begin{eqnarray}\label{varepsilon_k_n}\tilde{\varepsilon}_{\mathbf{k}_n}&\le& b_0 \left(\int_{k'_0-T-1}^{\infty}
\Phi^{\gamma}(t)\,dt+\int_{k'_0-1}^{\infty}\Phi^{\gamma}(t)\,dt\right)+\sum_{j=0}^{J-1}\sqrt{c_j}2^{\frac {j}2}\left(\int_{k_j-2^jT-1}^{\infty}
\Phi^{\gamma}(t)\,dt\right.\nonumber
\\&+&\left.\int_{k_j-1}^{\infty}\Phi^{\gamma}(t)\,dt\right)+\sum_{j=J}^{\infty}\sqrt{c_j}2^{\frac
{j}2} \left(3\Phi^{\gamma}(0)+4\int_{1/2}^{\infty}\Phi^{\gamma}(t)\,dt\right)=:\varepsilon_{\mathbf{k}_n}.\end{eqnarray}

The choice of $J$ and assumption (\ref{sum_cj}) imply that $\varepsilon_{\mathbf{k}_n}\to 0,$ when $n\to\infty,$ $k_0'\to\infty,$
and $k_j\to\infty$ for all $j\in \mathbb N_0\,.$

It is worth noticing that $\delta(\cdot)$ is an increasing function,
$I(\varepsilon_0)\le \delta(\varepsilon_0)$ for any $\varepsilon_0,$
and $\tilde{\varepsilon}_{\mathbf{k}_n}\le
\varepsilon_{\mathbf{k}_n}.$ Hence, for
$u>8\delta(\varepsilon_{\mathbf{k}_n})$ we get

$$\exp\left\{-\frac{\left(u-\sqrt{8uI(\tilde{\varepsilon}_{\mathbf{k}_n})}\right)^2}{2\tilde{\varepsilon}_{\mathbf{k}_n}^2}\right\}\le \exp\left\{-\frac{\left(u-\sqrt{8u\delta(\varepsilon_{\mathbf{k}_n})}\right)^2}{2\varepsilon_{\mathbf{k}_n}^2}\right\}\,.$$

Finally, an application of Lemmata~\ref{71} and \ref{lem7} completes
the proof.
 \end{proof}
 
 \begin{remk}
Note, that $\varepsilon_{\mathbf{k}_n}\to 0,$ if and only if $n,$ $k_0',$ and all $k_j, j\ge 0,$ approach infinity.
 \end{remk}
 
 \begin{remk}
 If $\varepsilon_{\mathbf{k}_n}\to 0$ then $\delta(\varepsilon_{\mathbf{k}_n})\to 0.$ Therefore the convergence in the theorem is exponential with the rate bounded by $\displaystyle 2\exp\left\{-\frac{{const}}{\varepsilon_{\mathbf{k}_n}^2}\right\},$ where one can choose $const\approx u^2$ for $\delta(\varepsilon_{\mathbf{k}_n})\ll u.$
 \end{remk}
 
 \begin{remk} In the theorem we only require that $k_0',$ and all $k_j, j\ge 0,$ approach infinity.
 If we narrow our general class of wavelet expansions $\mathbf{X}_{n,\mathbf{k}_n}(t)$ by specifying rates of growth of the sequences $\mathbf{k}_n$ we can enlarge classes of wavelets bases and random processes in the theorem and obtain explicit rates of convergence by specifying $\varepsilon_{\mathbf{k}_n}.$ 
 
 For instance, consider the examples in Section~\ref{sec5}. It was shown that  $\sqrt{c_j}\le C/{2^{j/2(1+2\varkappa)}}.$ Let the sequences $\mathbf{k}_n:=(k_0',k_0,...,k_{n-1})$ be chosen so that $\int_{k^*}^{\infty} \Phi^{\gamma}(t)\,dt\le C/{2^{J\varkappa}}$ for $k^*:=\min(k_0'-T,k_0-T,k_1-2T,...,k_{J-1}-2^{J-1}T)-1.$  Then, by (\ref{varepsilon_k_n}) we get  $\varepsilon_{\mathbf{k}_n}\le C/{2^{J\varkappa}}.$  
 \end{remk}
 
 \begin{remk} Lemma~\ref{lem7} and formula~(\ref{varepsilon_k_n}) provide simple expressions to computer $\varepsilon_{\mathbf{k}_n}$ and $\delta(\varepsilon_{\mathbf{k}_n}).$ It allows  specifying Theorem~\ref{th7} for various stochastic processes and wavelets.
  \end{remk}

\section{Acknowledgements} The research of first two authors  was partially supported by La Trobe University Research Grant "Sto\-chastic Approximation in Finance and Signal Processing."
The authors are grateful for the referee's comments, which helped to improve the style of the presentation.

\end{document}